\documentclass[sn-mathphys]{sn-jnl}

\usepackage{color}

\jyear{2021}%

\theoremstyle{thmstyleone}%
\newtheorem{theorem}{Theorem}%
\newtheorem{lemma}{Lemma}%
\newtheorem{corollary}{Corollary}%

\theoremstyle{thmstyletwo}%
\newtheorem{example}{Example}%
\newtheorem{remark}{Remark}%

\theoremstyle{thmstylethree}%

\def\rhs{g}
\newcommand\bb{{\boldsymbol b}}
\newcommand\RR{\mathbb{R}}
\newcommand\TT{{\mathscr T}}

\begin{document}

\title[On algebraically stabilized schemes for convection--diffusion--reaction problems]
{On algebraically stabilized schemes for convection--diffusion--reaction problems}

\author[1,2]{\fnm{Volker} \sur{John}}\email{john@wias-berlin.de, ORCID 0000-0002-2711-4409}

\author*[3]{\fnm{Petr} \sur{Knobloch}}\email{knobloch@karlin.mff.cuni.cz, ORCID 0000-0003-2709-5882}

\affil[1]{\orgname{Weierstrass Institute for Applied Analysis and
Stochastics (WIAS)},
\orgaddress{\street{Mohrenstr. 39}, \city{Berlin}, \postcode{10117},
\country{Germany}}}

\affil[2]{\orgdiv{Department of Mathematics and Computer Science},
\orgname{Freie Universit\"at Berlin},
\orgaddress{\street{Arnimallee 6}, \city{Berlin}, \postcode{14195},
\country{Germany}}}

\affil*[3]{\orgdiv{Department of Numerical Mathematics, Faculty of Mathematics
and Physics}, \orgname{Charles University},
\orgaddress{\street{Sokolovsk\'a 83}, \city{Praha 8}, \postcode{18675},
\country{Czech Republic}}}

\abstract{An abstract framework is developed that enables the analysis
of algebraically stabilized discretizations in a unified way. This framework
is applied to a discretization of this kind for convection--diffusion--reaction
equations. The definition of this scheme contains a new limiter that improves
a standard one in such a way that local and global discrete maximum principles 
are satisfied on arbitrary simplicial meshes.
}


\maketitle

This work has been
supported through the grant No.~19-04243S of the Czech Science Foundation.

\pagestyle{myheadings}
\thispagestyle{plain}
\markboth{V.~John, P.~Knobloch}{On algebraically stabilized schemes\dots}

\section{Introduction}
\label{s1}

The modeling of physical processes is usually performed on the basis of 
physical laws, like conservation laws. The derived model is physically 
consistent if its solutions satisfy the respective laws and, in addition, other 
important physical properties. Convection--diffusion--reaction equations, 
which will be considered in this paper, are the result of modeling 
conservation of scalar quantities, like temperature (energy balance)
or concentrations (mass balance). Besides conservation, bounds for the 
solutions of such equations can be proved (if the data satisfy certain 
conditions) that reflect physical properties, like non-negativity of 
concentrations or 
that the temperature is maximal on the boundary of the body if there are no 
heat sources and chemical processes inside the body. Such bounds 
are called maximum principles, e.g., see \cite{GT01}. A serious difficulty 
for computing numerical approximations of solutions of convection--diffusion--reaction
equations is that most proposed discretizations do not satisfy the discrete
counterpart of the maximum principles, so-called discrete maximum
principles (DMPs), and thus they are not physically consistent in this respect. 
One of the exceptions are algebraically stabilized 
finite element schemes, e.g., algebraic flux correction (AFC) schemes, 
where DMPs have been proved rigorously. Methods of this type will be studied in
this paper.

The theory developed in this paper is motivated by the numerical solution of 
the scalar steady-state convection--diffusion--reaction problem
\begin{equation}
   -\varepsilon\,\Delta u+{\bb}\cdot\nabla u+c\,u=\rhs\quad\mbox{in $\Omega$}\,,
    \qquad\qquad u=u_b\quad\mbox{on $\partial\Omega$}\,,\label{strong-steady}
\end{equation}
where $\Omega\subset{\mathbb R}^d$, $d\ge1$, is a bounded domain with a
Lipschitz-continuous boundary $\partial\Omega$ that is assumed to be polyhedral
(if $d\ge2$). Furthermore, the diffusion coefficient $\varepsilon>0$ is a 
constant and the convection field ${\bb}\in W^{1,\infty}(\Omega)^d$, the 
reaction field $c\in L^\infty(\Omega)$, the right-hand side 
$\rhs\in L^2(\Omega)$, and the Dirichlet boundary data 
$u_b\in H^{\frac{1}{2}}(\partial\Omega)\cap C(\partial\Omega)$ are given 
functions satisfying
\begin{equation}
   \nabla\cdot\bb=0\,,\qquad c\ge\sigma_0\ge0\qquad\quad\mbox{in $\Omega$}\,,
   \label{eq_ass_b_c}
\end{equation}
where $\sigma_0$ is a constant. 

In applications, one encounters typically the convection-dominated regime, 
i.e., it is $\varepsilon \ll L \|\bb\|_{0,\infty,\Omega}^{}$, where $L$ is a 
characteristic length scale of the problem and $\|\cdot\|_{0,\infty,\Omega}^{}$ 
denotes the norm in $L^\infty(\Omega)^d$. Then, a characteristic feature of 
(weak) solutions of \eqref{strong-steady}
is the appearance of layers, which are thin regions where the solution possesses a steep 
gradient. The thickness of layers is usually (much) below the resolution of the
mesh. 
It is well known that the standard Galerkin finite element method cannot cope with 
this situation and one has to utilize a so-called stabilized discretization, e.g., 
see \cite{RST08}.

Linear stabilized finite element methods that satisfy DMPs, usually with restrictions to the type
of mesh, like the upwind method from \cite{BT81}, compute in general very inaccurate results
with strongly smeared layers. In order to compute accurate solutions, a nonlinear 
method has to be applied, typically with parameters that depend on the concrete numerical 
solution. A nonlinear upwind method was proposed in \cite{MH85} and improved in \cite{Knobloch06}. In \cite{BE05}, a nonlinear edge stabilization method was proposed, 
see \cite{BE02,BH04} for related methods, for which a DMP was proved providing that 
a certain discretization parameter is chosen to be sufficiently large and the mesh is of a certain 
type. However, already the numerical results presented in \cite{BE05} show spurious 
oscillations. Our own experience from \cite{JK08} is that the nonlinear problems for 
sufficiently large parameters often cannot be solved numerically. 

A class of methods that has been developed intensively in recent years is the class of algebraically
stabilized schemes, e.g., see \cite{BB17,BJK17,GNPY14,Kuzmin06,Kuzmin07,Kuzmin09,Kuzmin12,KuzminMoeller05,KS17,KT04,LKSM17}.
The origins
of this approach can be tracked back to \cite{BorisBook73,Zalesak79}.
In these schemes, the stabilization is performed on the basis of 
the algebraic system of equations obtained with the Galerkin finite element method. 
Then, so-called limiters are computed, which maintain the conservation property and
which restrict the stabilized discretization mainly to a vicinity of layers to ensure the satisfaction 
of DMPs without compromising the accuracy. There are several limiters proposed in the literature, like the so-called 
Kuzmin~\cite{Kuzmin07}, BJK~\cite{BJK17}, or BBK~\cite{BBK17} limiters. 
Both, the Kuzmin and the BBK limiters
were utilized in \cite{BBK17a} for defining a scheme that blends a standard linear 
stabilized scheme in smooth regions and a nonlinear stabilized method in a vicinity of layers. 

An advantage of algebraically stabilized schemes is that they satisfy a DMP by
construction, often under some assumptions on the mesh, and they usually 
provide sharp approximations of layers, cf.~the numerical results in, e.g.,
\cite{ACF+11,GKT12,JS08,Kuzmin12}. In numerical 
studies presented in \cite{JJ19}, it turned out that the results with the BJK limiter 
were usually more accurate than with the Kuzmin limiter, if the nonlinear problems 
for the BJK limiter could be solved. However, solving these problems was often not 
possible for strongly convection-dominated cases. Numerical studies in \cite{BJK18} show 
that using the Kuzmin limiter leads to solutions with sharper layers compared with the 
solutions obtained with the BBK limiter.
As a consequence of these experiences, it seems to be advisable from the point 
of view of applications to use algebraically stabilized schemes on the basis of the 
Kuzmin limiter. The AFC scheme with the Kuzmin limiter was analyzed in 
\cite{BJK16}, thereby proving the existence of a solution, the satisfaction 
of a local DMP, and an error estimate. The local DMP requires
 lumping the reaction term and using
certain types of meshes,
e.g., Delaunay meshes in two dimensions, 
analogously as for the methods from \cite{BBK17,BBK17a}. 

The conservation and stability properties of algebraically stabilized schemes 
are given if the 
added stabilization is a symmetric term. For many schemes, this term consists 
of two factors, an artificial diffusion matrix and the matrix of the limiters, 
and usually the methods are constructed in such a way that both factors are 
symmetric. Only recently, motivated by \cite{BB17}, a more general approach 
where only the product is symmetric but not the individual factors was 
considered in \cite{Kno21}. 

The first main goal of this paper is the development of an abstract framework 
that allows to analyze algebraically stabilized discretizations in a unified 
way. Although our main interest is the numerical solution of problem
\eqref{strong-steady}, many considerations will be more general and then 
problem \eqref{strong-steady} and its discretizations will only serve as 
a motivation for our assumptions.
Hence, this framework covers a larger class of algebraically stabilized discretizations than the available 
analysis. 

The second main goal  consists in proposing and analyzing a modification of the 
Kuzmin limiter such that, if applied in the framework of the algebraic
stabilization of \cite{Kno21}, the positive features of the AFC method with the
Kuzmin limiter are preserved on meshes where it works well and, in addition, 
local and global DMPs can be proved on arbitrary simplicial meshes. In 
particular, our intention was to preserve the upwind character of the AFC 
method with the Kuzmin limiter. There are already proposals in this direction
in the framework of AFC methods. In \cite{Kno17}, the Kuzmin limiter is 
replaced in cases where it does not lead to the validity of the local DMP 
in a somewhat ad hoc way by a value that introduces more artificial diffusion.
The satisfaction of the local DMP on arbitrary simplicial meshes could be 
proved for this approach. Whether or not the assumption for the existence of a 
solution of the nonlinear problem is satisfied with this limiter is not 
discussed. A combination of the Kuzmin and the BJK limiters to obtain a limiter
of upwind type for which the AFC scheme satisfies a local DMP on arbitrary 
simplicial meshes and is linearity preserving is proposed in \cite{Kno19}. The 
definition of this limiter is closer to the BJK than to the Kuzmin limiter. As 
already mentioned, in \cite{Kno21}, a new algebraically stabilized method was
proposed that does not require the symmetry of the limiter. Initial numerical 
results for a nonsymmetric  modification of the Kuzmin limiter are presented 
in \cite{Kno21}, but a numerical analysis is missing. The abstract framework 
mentioned in the previous paragraph covers in particular the method 
from \cite{Kno21}.

In the present paper, the limiter from \cite{Kno21} is written in a simpler 
form, without using internodal fluxes typical for AFC methods. Moreover, a 
novel modification is performed that improves the accuracy in some computations 
using non-Delaunay meshes. Of course, this modification is performed in such a
way that the resulting method still fits in the abstract analytic framework. 
The definition of the new method does not
contain any ambiguity, in contrast to the AFC method with Kuzmin limiter, which 
is not uniquely defined in some cases (cf.~Remark~8 in \cite{BJK16}). A further 
advantage of the considered approach is that, in contrast to the AFC method 
with Kuzmin limiter, lumping of the reaction term is no longer necessary 
for the satisfaction of the DMP, which enables to obtain sharper layers as we 
will demonstrate by numerical results.

This paper is organized as follows. Sect.~\ref{s2} introduces the basic 
discretization of \eqref{strong-steady} and its algebraic form. An abstract 
framework for an algebraic stabilization is presented in Sect.~\ref{s3}. The 
following section studies the solvability and the satisfaction of local and 
global DMPs for the abstract algebraic stabilization and Sect.~\ref{s5} 
provides an error analysis. In Sect.~\ref{s6}, the AFC scheme with Kuzmin 
limiter as an example of algebraic stabilization from Sect.~\ref{s3} is 
presented, its properties are discussed for the discretizations from
Sect.~\ref{s2} and the definition of the limiter is reformulated. The 
reformulation is utilized in Sect.~\ref{s7} for proposing a new limiter such 
that the resulting algebraically stabilized scheme is of upwind type and 
satisfies DMPs on arbitrary simplicial meshes. Sect.~\ref{numerics} presents 
numerical examples which show that the algebraically stabilized scheme with the 
new limiter in fact cures the deficiencies of the AFC scheme with Kuzmin 
limiter. 

\section{The convection--diffusion--reaction problem and its finite element
discretization}
\label{s2}

The weak solution of the convection--diffusion--reaction problem 
\eqref{strong-steady} is a function $u\in H^1(\Omega)$ 
satisfying the boundary condition $u=u_b$ on $\partial\Omega$ and the
variational equation
\begin{equation*}
   a(u,v)=(\rhs,v)\qquad\forall\,\,v\in H^1_0(\Omega)\,,
\end{equation*}
where
\begin{equation}
   a(u,v)=\varepsilon\,(\nabla u,\nabla v)+({\bb}\cdot\nabla u,v)+(c\,u,v)\,.
   \label{eq_form_a}
\end{equation}
As usual, $(\cdot,\cdot)$ denotes the inner product in $L^2(\Omega)$ or 
$L^2(\Omega)^d$. It is well known that the weak solution of 
\eqref{strong-steady} exists and is unique (cf. \cite{Evans}).

An important property of problem \eqref{strong-steady} is that, for $c\ge0$ 
in $\Omega$, its solutions satisfy the maximum principle. The classical
maximum principle (cf.~\cite{Evans}) states the following: if $u\in
C^2(\Omega)\cap C(\overline\Omega)$ solves \eqref{strong-steady} and the
functions $\bb$ and $c$ are bounded in~$\Omega$, then, for any set
$G\subset\overline\Omega$, one has the implications
\begin{align}
   &\rhs\le0\quad\mbox{\rm in}\,\,\,G\qquad\Rightarrow\qquad
    \max_{\overline G}\,u\le\max_{\partial G}\,u^+\,,\label{eq_max_1a}\\
   &\rhs\ge0\quad\mbox{\rm in}\,\,\,G\qquad\Rightarrow\qquad
    \min_{\overline G}\,u\ge\min_{\partial G}\,u^-\,,\label{eq_max_1b}
\end{align}
where $u^+=\max\{u,0\}$, $u^-=\min\{u,0\}$. If, in addition, $c=0$ in $G$, then
\begin{align}
   &\rhs\le0\quad\mbox{\rm in}\,\,\,G\qquad\Rightarrow\qquad
    \max_{\overline G}\,u=\max_{\partial G}\,u\,,\label{eq_max_2a}\\
   &\rhs\ge0\quad\mbox{\rm in}\,\,\,G\qquad\Rightarrow\qquad
    \min_{\overline G}\,u=\min_{\partial G}\,u\,.\label{eq_max_2b}
\end{align}
Analogous statements also hold for the weak solutions, cf.~\cite{GT01}.

To define a finite element discretization of problem \eqref{strong-steady}, we
consider a simplicial triangulation $\TT_h$ of $\overline\Omega$ which is 
assumed to belong to a regular family of triangulations in the sense of
\cite{Ciarlet}. Furthermore, we introduce finite element spaces
\begin{displaymath}
   W_h=\{v_h\in C(\overline{\Omega})\,;\,\,v_h\vert _T^{}\in \mathbb{P}_1(T)\,\,
   \forall\, T\in \TT_h\}\,,\qquad V_h=W_h\cap H^1_0(\Omega)\,,
\end{displaymath}
consisting of continuous piecewise linear functions. The vertices of the
triangulation $\TT_h$ will be denoted by $x_1,\dots,x_N$ and we assume that
$x_1,\dots,x_M\in\Omega$ and $x_{M+1},\dots,x_N\in\partial\Omega$. Then the
usual basis functions $\varphi_1,\dots,\varphi_N$ of $W_h$ are defined by the
conditions $\varphi_i(x_j)=\delta_{ij}$, $i,j=1,\dots,N$, where $\delta_{ij}$
is the Kronecker symbol. Obviously, the functions $\varphi_1,\dots,\varphi_M$ 
form a basis of $V_h$. Any function $u_h\in W_h$ can be written in a unique way
in the form
\begin{equation}
   u_h=\sum_{i=1}^Nu_i\,\varphi_i\label{eq_vect_fcn_identification}
\end{equation}
and hence it can be identified with the coefficient vector 
${\rm U}=(u_1,\dots,u_N)$.

Now an approximate solution of problem \eqref{strong-steady} can be
introduced as the solution of the following finite-dimensional problem:

\vspace*{2mm}

Find $u_h\in W_h$ such that $u_h(x_i)=u_b(x_i)$, $i=M+1,\dots,N$, and
\begin{equation}
   a_h(u_h,v_h)=(\rhs,v_h)\qquad\forall\,\,v_h\in V_h\,,\label{v3}
\end{equation}

\vspace*{1mm}

\noindent
where $a_h$ is a bilinear form approximating the bilinear form $a$. In 
particular, one can use $a_h=a$. Another possibility is to set
\begin{equation}
   a_h(u_h,v_h)=\varepsilon\,(\nabla u_h,\nabla v_h)+({\bb}\cdot\nabla u_h,v_h)
   +\sum_{i=1}^M\,(c,\varphi_i)\,u_i\,v_i\label{eq_lumped_react}
\end{equation}
for any $u_h\in W_h$ and $v_h\in V_h$, i.e., to consider a lumping of the
reaction term $(c\,u_h,v_h)$ in $a(u_h,v_h)$. This may help to satisfy the DMP
for problem \eqref{v3}, cf.~Sect.~\ref{s6}. We 
assume that $a_h$ is elliptic on the space $V_h$, i.e., there is a constant 
$C_a>0$ such that
\begin{equation}
   a_h(v_h,v_h)\ge C_a\,\| v_h\| _a^2\qquad\forall\,\,v_h\in V_h\,,\label{v1}
\end{equation}
where $\| \cdot\| _a^{}$ is a norm on the space $H^1_0(\Omega)$ but generally 
only a seminorm on the space $H^1(\Omega)$. This guarantees that the discrete
problem \eqref{v3} has a unique solution. In view of \eqref{eq_ass_b_c},
the ellipticity condition \eqref{v1} holds for both $a_h=a$ and $a_h$ defined
by \eqref{eq_lumped_react} with $C_a=1$ and 
\begin{equation}\label{eq:cont_norm}
   \| v\| _a^2=\varepsilon\,\vert v\vert _{1,\Omega}^2+\sigma_0\,\| v\| _{0,\Omega}^2\,.
\end{equation}

We denote
\begin{alignat}{2}
   a_{ij}&=a_h(\varphi_j,\varphi_i)\,,\qquad &&i,j=1,\dots,N\,,\label{13}\\
   \rhs_i&=(\rhs,\varphi_i)\,,\qquad &&i=1,\dots,M\,,\label{14}\\
   u^b_i&=u_b(x_i)\,,\qquad &&i=M+1,\dots,N\,.\label{15}
\end{alignat}
Then $u_h$ is a solution of the finite-dimensional problem \eqref{v3} if and 
only if the coefficient vector $(u_1,\dots,u_N)$ corresponding to $u_h$ 
satisfies the algebraic problem
\begin{align*}
   &\sum_{j=1}^N\,a_{ij}\,u_j=\rhs_i\,,\qquad i=1,\dots,M\,,\\
   &u_i=u^b_i\,,\qquad i=M+1,\dots,N\,.
\end{align*}

As discussed in the introduction, the above discretizations are not appropriate 
in the convection-dominated regime and a stabilization has to be applied. In
the next sections, algebraic stabilization techniques will be studied. As 
already mentioned, a general framework will be presented and the numerical 
solution of convection--diffusion--reaction equations serves just as a 
motivation for the assumptions.

\section{An abstract framework}
\label{s3}

In this section we assume that we are given a system of linear algebraic
equations of the form
\begin{align}
   &\sum_{j=1}^N\,a_{ij}\,u_j=\rhs_i\,,\qquad i=1,\dots,M\,,\label{21}\\
   &u_i=u^b_i\,,\qquad i=M+1,\dots,N\,,\label{21b}
\end{align}
(with $0<M<N$) corresponding to a discretization of a linear boundary value
problem for which the maximum principle holds. An example is the algebraic
problem derived in the preceding section. 

We assume that the row sums of the system matrix are nonnegative, i.e.,
\begin{equation}\label{mp2}
   \sum_{j=1}^N\,a_{ij}\ge0\,,\qquad\quad i=1,\dots,M\,,
\end{equation}
and that the submatrix $(a_{ij})_{i,j=1}^M$ is positive definite, i.e.,
\begin{equation}
   \sum_{i,j=1}^M\,u_i\,a_{ij}\,u_j>0\qquad
   \forall\,\,(u_1,\dots,u_M)\in{\mathbb R}^M\setminus\{0\}\,.\label{32}
\end{equation}
For the discretizations from the previous section, the latter property follows
from \eqref{v1}, whereas \eqref{mp2} is a consequence of the nonnegativity of
$c$ and the fact that $\sum_{j=1}^N\varphi_j=1$.

Since the algebraic problem \eqref{21}, \eqref{21b} is assumed to approximate 
a problem satisfying the maximum principle, it is natural to require that an
analog of this property also holds in the discrete case, at least locally. 
Then an important physical property of the original problem will be preserved 
and spurious oscillations of the approximate solution will be excluded. To 
formulate a local DMP, we have to specify a neighborhood
\begin{displaymath}
   S_i\subset\{1,\dots,N\}\setminus\{i\}
\end{displaymath}
of any $i\in\{1,\dots,M\}$ (i.e., of any interior vertex $x_i$ if the geometric
interpretation from the previous section is considered). For example, one can
set
\begin{equation}\label{def-S1}
   S_i=\{j\in\{1,\dots,N\}\setminus\{i\}\,;\,\, a_{ij}\neq0\}\,,\qquad 
   i=1,\dots,M\,.
\end{equation}
Then, under the assumptions \eqref{mp2} and \eqref{32}, the solution of
\eqref{21}, \eqref{21b} satisfies the local DMP
\begin{equation}
   g_i\le 0\quad\Rightarrow\quad u_i\le\max_{j\in S_i}\,u_j^+\,,\qquad\qquad
   g_i\ge 0\quad\Rightarrow\quad u_i\ge\min_{j\in S_i}\,u_j^-\label{dmp1}
\end{equation}
(with any $i\in\{1,\dots,M\}$) if and only if (cf.~\cite[Lemma 21]{BJK16})
\begin{equation}
   a_{ij}\le0\qquad\forall\,\,i\neq j,\,i=1,\dots,M,\,j=1,\dots,N\,.\label{12}
\end{equation}
Moreover, the stronger local DMP
\begin{equation}
   g_i\le 0\quad\Rightarrow\quad u_i\le\max_{j\in S_i}\,u_j\,,\qquad\qquad
   g_i\ge 0\quad\Rightarrow\quad u_i\ge\min_{j\in S_i}\,u_j\label{dmp2}
\end{equation}
holds (again with any $i\in\{1,\dots,M\}$) if and only if the conditions 
\eqref{12} and
\begin{equation}\label{mp}
   \sum_{j=1}^N\,a_{ij}=0\,,\qquad\quad i=1,\dots,M\,,
\end{equation}
are satisfied (cf.~\cite[Lemma 22]{BJK16}). For the discretizations from the 
previous section, \eqref{mp} holds if $c=0$ in $\Omega$, which is a condition
used for proving \eqref{eq_max_2a} and \eqref{eq_max_2b}, i.e., a counterpart
of \eqref{dmp2}.

In many cases, the condition \eqref{12} is violated (like for the
discretizations from the previous section in the convection-dominated regime) 
and hence the local DMPs \eqref{dmp1} and \eqref{dmp2} do not hold. To enforce 
the DMP, one can add a sufficient amount of artificial diffusion to \eqref{21}, 
e.g., in the following way. First, the system matrix is extended to a matrix 
${\mathbb A}=(a_{ij})_{i,j=1}^N$, typically using the matrix corresponding to 
the underlying discretization in the case when homogeneous natural boundary 
conditions are used instead of the Dirichlet ones (i.e., using \eqref{13} if 
the setting of the previous section is considered). Then one can define a 
symmetric artificial diffusion matrix $\mathbb D=(d_{ij})_{i,j=1}^N$ possessing 
the entries
\begin{equation}\label{def-dij}
   d_{ij}=d_{ji}=-\max\{a_{ij},0,a_{ji}\}\qquad\forall\,\,i\neq j\,,\qquad\qquad
   d_{ii}=-\sum_{j\neq i}\,d_{ij}\,.
\end{equation}
The matrix $\mathbb D$ has zero row and column sums and is positive semidefinite
(cf.~\cite[Lemma 1]{BJK16}), the matrix ${\mathbb A}+{\mathbb D}$ has 
nonpositive off-diagonal entries by construction and the submatrix 
$(a_{ij}+d_{ij})_{i,j=1}^M$
is positive definite. Consequently, the stabilized algebraic problem
\begin{align}
   &\sum_{j=1}^N\,(a_{ij}+d_{ij})\,u_j=\rhs_i\,,\qquad i=1,\dots,M\,,
   \label{22}\\
   &u_i=u^b_i\,,\qquad i=M+1,\dots,N\,,\label{22b}
\end{align}
is uniquely solvable and its solution satisfies the local DMP \eqref{dmp1} with
\begin{equation}\label{ass_S_1}
   S_i=\{j\in\{1,\dots,N\}\setminus\{i\}\,;\,\,
             a_{ij}\neq0\,\,\,\mbox{or}\,\,\,a_{ji}>0\}\,,\qquad i=1,\dots,M\,.
\end{equation}
If the condition \eqref{mp} holds, then the solution of \eqref{22}, \eqref{22b} 
also satisfies the stronger local DMP \eqref{dmp2}, again with $S_i$ defined by 
\eqref{ass_S_1}. Moreover, if the above stabilization is applied to the
discretizations from the previous section, then, for weakly acute
triangulations, the approximate solutions converge to the solution of
\eqref{strong-steady}, see \cite{BJK16}.

However, the amount of artificial diffusion added in \eqref{22} is usually too
large and leads to an excessive smearing of layers if it is applied to
stabilize discretizations of \eqref{strong-steady} in the convection-dominated 
regime. To suppress the smearing, the artificial diffusion should be added 
mainly in regions where the solution changes abruptly and hence it should 
depend on the unknown approximate solution ${\rm U}=(u_1,\dots,u_N)$. This
motivates us to introduce a general artificial diffusion matrix 
${\mathbb B}({\rm U})=(b_{ij}({\rm U}))_{i,j=1}^N$ having analogous properties as the matrix
${\mathbb D}$, i.e., for any ${\rm U}\in{\mathbb R}^N$, we assume that
\begin{alignat}{2}
   &b_{ij}({\rm U})=b_{ji}({\rm U})\,,\qquad\quad&&i,j=1,\dots,N\,,
   \label{eq_b1}\\[2mm]
   &b_{ij}({\rm U})\le0,\qquad\quad&&i,j=1,\dots,N\,,\,\,i\neq j\,,
   \label{eq_b2}\\
   &\sum_{j=1}^N\,b_{ij}({\rm U})=0\,,\qquad\quad&&i=1,\dots,N\,.\label{eq_b3}
\end{alignat}
Like above, we introduce local index sets $S_i$ such that
\begin{equation}\label{eq_gen_S}
   \{j\in\{1,\dots,N\}\setminus\{i\}\,;\,\, a_{ij}\neq0\}\subset
   S_i\subset\{1,\dots,N\}\setminus\{i\}\,,\quad
   i=1,\dots,M\,,
\end{equation}
and, for any ${\rm U}\in{\mathbb R}^N$,
\begin{equation}
   b_{ij}({\rm U})=0\qquad
   \forall\,\,j\not\in S_i\cup\{i\},\,i=1,\dots,M\,.\label{eq_b4}
\end{equation}
Let us mention that if the algebraic problem \eqref{21}, \eqref{21b}
corresponds to a finite element discretization based on piecewise linear
functions as in the preceding section, one can usually use index sets
\begin{equation}\label{def_S}
   S_i=\{j\in\{1,\dots,N\}\setminus\{i\}\,;\,\,\mbox{\rm $x_i$ and $x_j$ are
         end points of the same edge}\}\,,
\end{equation}
$i=1,\dots,M$, where $x_1,\dots,x_N$ are the vertices of the underlying 
simplicial triangulation, numbered as in the preceding section.

Now, we consider the nonlinear algebraic problem
\begin{align}
   &\sum_{j=1}^N\,(a_{ij}+b_{ij}({\rm U}))\,u_j=\rhs_i\,,\qquad i=1,\dots,M\,,
   \label{23}\\
   &u_i=u^b_i\,,\qquad i=M+1,\dots,N\,.\label{23b}
\end{align}
Note that, in view of \eqref{eq_b3} and \eqref{eq_b4}, system \eqref{23}
can be written in the form
\begin{equation}
   \sum_{j=1}^N\,a_{ij}\,u_j+\sum_{j\in S_i}\,b_{ij}({\rm U})\,(u_j-u_i)
   =\rhs_i\,,\qquad i=1,\dots,M\,.\label{23c}
\end{equation}
In view of \eqref{eq_b1} and \eqref{eq_b2}, one obtains the important property
(cf.~\cite[Lemma 1]{BJK16})
\begin{equation}
   \sum_{i,j=1}^N\,v_i\,b_{ij}({\rm U})\,(v_j-v_i)
   =-\frac12\,\sum_{i,j=1}^N\,b_{ij}({\rm U})\,(v_j-v_i)^2\ge0
   \qquad\forall\,\,{\rm U},{\rm V}\in{\mathbb R}^N\,.\label{eq_pos_semidef}
\end{equation}
Thus, due to \eqref{eq_b3}, the matrix ${\mathbb B}({\rm U})$ is positive
semidefinite for any ${\rm U}\in{\mathbb R}^N$.

\section{Analysis of the abstract nonlinear algebraic problem}
\label{s4}

The aim of this section is to investigate the solvability and the validity of
the DMP for the nonlinear algebraic problem \eqref{23}, \eqref{23b}. These
investigations will generalize the results obtained in
\cite{BJK16,Kno17,BJK18}.

To prove the solvability of the system \eqref{23}, \eqref{23b}, we make the
following assumption.

\vspace*{1ex}
\noindent{\bf Assumption (A1): } For any $i\in\{1,\dots,M\}$ and any
$j\in\{1,\dots,N\}$, the function $b_{ij}({\rm U})(u_j-u_i)$ is a continuous 
function of ${\rm U}=(u_1,\dots,u_N)\in{\mathbb R}^N$ and, for any
$i\in\{1,\dots,M\}$ and any $j\in\{M+1,\dots,N\}$, the function 
$b_{ij}({\rm U})$ is a bounded function of ${\rm U}\in{\mathbb R}^N$.
\vspace*{1ex}

\begin{theorem}\label{existence}
Let \eqref{32} and \eqref{eq_b1}--\eqref{eq_b3} hold and let Assumption (A1)
be satisfied. Then there exists a solution of the nonlinear problem
\eqref{23}, \eqref{23b}.
\end{theorem}

\begin{proof} The proof follows the lines of the proof of Theorem 3 in
\cite{BJK16}. We denote by $\widetilde{\rm V}\equiv(v_1,\dots,v_M)$  the
elements of the space ${\mathbb R}^M$ and, if $v_i$ with $i\in\{M+1,\dots,N\}$
occurs, we assume that $v_i=u^b_i$. To any 
$\widetilde{\rm V}\in{\mathbb R}^M$, we assign ${\rm V}:=(v_1,\dots,v_N)$.
Let us define the operator $T:{\mathbb R}^M\to{\mathbb R}^M$ by
\begin{displaymath}
   (T\,\widetilde{\rm V})_i=\sum_{j=1}^N\,a_{ij}\,v_j
   +\sum_{j=1}^N\,b_{ij}({\rm V})\,(v_j-v_i)-\rhs_i\,,\qquad i=1,\dots,M\,.
\end{displaymath}
Then ${\rm U}$ is a solution of the nonlinear problem \eqref{23}, \eqref{23b}
if and only if $T\,\widetilde{\rm U}=0$. The operator $T$ is continuous and, 
in view of \eqref{32} and \eqref{eq_pos_semidef}, there exist constants $C_1$,
$C_2>0$ such that (cf.~\cite[Theorem 3]{BJK16} for details)
\begin{displaymath}
   (T\,\widetilde{\rm V},\widetilde{\rm V})\ge C_1\,\| \widetilde{\rm V}\| ^2-C_2
   \qquad\forall\,\,\widetilde{\rm V}\in {\mathbb R}^M\,,
\end{displaymath}
where $(\cdot,\cdot)$ is the usual inner product in ${\mathbb R}^M$ and 
$\| \cdot\| $ the corresponding (Euclidean) norm.
Then, for any $\widetilde{\rm V}\in{\mathbb R}^M$ satisfying 
$\| \widetilde{\rm V}\| =\sqrt{2\,C_2/C_1}$, one has 
$(T\,\widetilde{\rm V},\widetilde{\rm V})>0$ and hence it follows from 
Brouwer's fixed-point theorem (see \cite[p.~164, Lemma 1.4]{Temam77}) that
there exists $\widetilde{\rm U}\in{\mathbb R}^M$ such that 
$T\,\widetilde{\rm U}=0$.
\end{proof}

\begin{remark}
For proving the solvability of \eqref{23}, \eqref{23b}, it would be sufficient 
to assume that the functions $b_{ij}({\rm U})u_j$ are continuous. However, 
since $b_{ij}({\rm U})$ should depend on local variations of $\rm U$ with
respect to $u_i$, the assumed continuity of $b_{ij}({\rm U})(u_j-u_i)$ is more 
useful. The functions $b_{ij}({\rm U})$ themselves are often not continuous,
cf.~Remark~\ref{remark_continuity}.
\end{remark}

\begin{remark}
The solution of \eqref{23}, \eqref{23b} is unique if 
${\mathbb B}({\rm U})\rm U$ is Lipschitz--continuous with a sufficiently small
constant. As pointed out in \cite{Lohman19}, this condition can be further
refined by introducing a positive semidefinite matrix $\mathbb D$, e.g., the
one defined in \eqref{def-dij}, and investigating the Lipschitz continuity of
$({\mathbb B}({\rm U})-\mathbb D){\rm U}$. Since, in view of \eqref{32}, there 
is $C>0$ such that
\begin{displaymath}
   C\,\| {\rm V}\| \le\| (\mathbb A+\mathbb D){\rm V}\| \qquad
   \forall\,\,{\rm V}\in{\mathbb R}^N, v_{M+1}=\dots=v_N=0\,,
\end{displaymath}
($\| \cdot\| $ is again the Euclidean norm on ${\mathbb R}^M$), the smallness
assumption on the Lipschitz constant can be expressed by the inequality
\begin{align}
   &\| ({\mathbb B}({\rm U})-\mathbb D){\rm U}
    -({\mathbb B}({\rm V})-\mathbb D){\rm V}\| 
   <\| (\mathbb A+\mathbb D)({\rm U-V})\| \nonumber\\
   &\hspace*{20mm}\forall\,\,{\rm U}\neq{\rm V}\in{\mathbb R}^N\,\,\,
   \mbox{\rm with}\,\,\,(u_{M+1},\dots,u_N)=(v_{M+1},\dots,v_N)\,.
   \label{eq_small_lipschitz}
\end{align}
Then, if ${\rm U}\neq\overline{\rm U}$ are two solutions of
\eqref{23}, \eqref{23b}, one has 
\begin{displaymath}
   [(\mathbb A+{\mathbb B}({\rm U})){\rm U}]_i
   =[(\mathbb A+{\mathbb B}(\overline{\rm U}))\overline{\rm U}]_i,\qquad
   i=1,\dots,M\,,
\end{displaymath}
and \eqref{eq_small_lipschitz} leads to a contradiction. Nevertheless, the
inequality \eqref{eq_small_lipschitz} is often not satisfied and then the
uniqueness of the nonlinear problem \eqref{23}, \eqref{23b} is open.
\end{remark}

Now let us investigate the validity of DMPs for
problem \eqref{23}, \eqref{23b}. To this end, one has to relate the properties
of the artificial diffusion matrix ${\mathbb B}({\rm U})$ to the matrix
${\mathbb A}$. This can be done in various ways and we shall use the following
assumption that generalizes the one used in \cite{Kno17}.

\vspace*{1ex}
\noindent{\bf Assumption (A2): } Consider any ${\rm U}=(u_1,\dots,u_N)\in\RR^N$ 
and any $i\in\{1,\dots,M\}$. If $u_i$ is a strict local extremum of $\rm U$ 
with respect to $S_i$ from \eqref{eq_gen_S}, \eqref{eq_b4}, i.e.,
\begin{displaymath}
   u_i>u_j\quad\forall\,\,j\in S_i\qquad\mbox{or}\qquad
   u_i<u_j\quad\forall\,\,j\in S_i\,,
\end{displaymath}
then
\begin{displaymath}
   a_{ij}+b_{ij}({\rm U})\le0\qquad\forall\,\,j\in S_i\,.
\end{displaymath}
\vspace*{1ex}

\begin{remark}
In contrast to linear problems, it is only assumed that off-diagonal entries of 
the matrix ${\mathbb A}+{\mathbb B}({\rm U})$ are nonpositive in rows
corresponding to indices where strict local extrema of $\rm U$ appear. If 
${\mathbb B}$ does not depend on $\rm U$, then Assumption (A2) implies that the
first $M$ rows of ${\mathbb A}+{\mathbb B}$ have nonpositive off-diagonal
entries, which is a necessary and sufficient condition for the validity of the
local DMP under our assumptions on ${\mathbb A}$ and ${\mathbb B}$.
\end{remark}

\begin{theorem}\label{thm_local_DMP}
Let \eqref{mp2}, \eqref{32}, and \eqref{eq_b1}--\eqref{eq_b4} hold and let
Assumption (A2) be satisfied. Then any solution 
${\rm U}=(u_1,\dots,u_N)\in\RR^N$ of \eqref{23} satisfies the local DMP
\eqref{dmp1} for all $i=1,\dots,M$. If condition \eqref{mp} holds, then the
stronger local DMP \eqref{dmp2} is also valid.
\end{theorem}

\begin{proof} 
The proof is basically the same as in \cite{Kno17}. Since it is short, we
repeat it for completeness. Let ${\rm U}=(u_1,\dots,u_N)\in\RR^N$ satisfy 
\eqref{23}. Consider any $i\in\{1,\dots,M\}$ and let $g_i\le 0$. Denoting
$A_i=\sum_{j=1}^Na_{ij}$, it follows from \eqref{23c} that
\begin{equation}\label{8a}
   A_i\,u_i+\sum_{j\in S_i}\,[a_{ij}+b_{ij}({\rm U})]\,(u_j-u_i)=\rhs_i\,.
\end{equation}
If $A_i>0$, we want to prove the first implication in \eqref{dmp1} for which
it suffices to consider $u_i>0$ since otherwise the implication trivially 
holds. If $A_i=0$, an arbitrary sign of $u_i$ is considered. Let us assume 
that $u_i>u_j$ for 
all $j\in S_i$. Then Assumption~(A2) implies that each term of the sum in 
\eqref{8a} is nonnegative. If $A_i=0$, then there is $j\in S_i$ such that 
$a_{ij}<0$ since $a_{ii}>0$ (see \eqref{32}). This together with \eqref{eq_b2}
implies that the sum in \eqref{8a} is positive. If $A_i>0$, then $A_i\,u_i>0$.
Thus, in both cases, the left-hand side of \eqref{8a} is positive, which is 
a contradiction. Therefore, there is $j\in S_i$ such that $u_i\le u_j$, which 
proves the first implication in \eqref{dmp2} and hence also in \eqref{dmp1}.
The statements for $g_i\ge0$ follow in an analogous way.
\end{proof}

Our next aim will be to show that, under the above assumptions, also a global
DMP is satisfied. First we prove the following general form of the DMP, which
generalizes a result proved in \cite{BJK18}. 

\begin{theorem}\label{thm:general_DMP}
Let \eqref{mp2}, \eqref{32}, and \eqref{eq_b1}--\eqref{eq_b4} hold and let
Assumptions (A1) and (A2) be satisfied. Consider any nonempty set 
$R\subset\{1,\dots,M\}$ and denote
\begin{equation}\label{eq_P_Q}
   P:=R\cup\bigcup_{i\in R}\,S_i\,,\qquad Q:=P\setminus R\,.
\end{equation}
Assume that $Q\neq\emptyset$. Then any solution 
${\rm U}=(u_1,\dots,u_N)\in\RR^N$ of \eqref{23} satisfies the DMP
\begin{eqnarray}
   &&g_i\le 0\quad\forall\,\,i\in R\qquad\Rightarrow\qquad 
   \max_{i\in P}u_i\le\max_{i\in Q}\,u_i^+\,,\label{dmp3a}\\
   &&g_i\ge 0\quad\forall\,\,i\in R\qquad\Rightarrow\qquad 
   \min_{i\in P}u_i\ge\min_{i\in Q}\,u_i^-\,.\label{dmp3b}
\end{eqnarray}
If, in addition, 
\begin{equation}\label{mp3}
   \sum_{j=1}^N\,a_{ij}=0\quad\,\,\quad\forall\,\,i\in R\,,
\end{equation}
then
\begin{eqnarray}
   &&g_i\le 0\quad\forall\,\,i\in R\qquad\Rightarrow\qquad 
   \max_{i\in P}u_i=\max_{i\in Q}\,u_i\,,\label{dmp4a}\\
   &&g_i\ge 0\quad\forall\,\,i\in R\qquad\Rightarrow\qquad 
   \min_{i\in P}u_i=\min_{i\in Q}\,u_i\,.\label{dmp4b}
\end{eqnarray}
\end{theorem}

\begin{proof} 
The proof is based on the technique used in \cite[Theorems 5.1 and 5.2]{Kno10}.
Let ${\rm U}=(u_1,\dots,u_N)$
satisfy \eqref{23} and let $g_i\le 0$ for all $i\in R$. We denote
\begin{displaymath}
   \widetilde{a}_{ij}:=a_{ij}+b_{ij}({\rm U})\,,\qquad
   i=1,\dots,M,\,\,j=1,\dots,N\,.
\end{displaymath}
Then, according to \eqref{eq_b3}--\eqref{eq_b4}, \eqref{mp2}, 
\eqref{eq_pos_semidef}, \eqref{32} and \eqref{23}, one has
\begin{align}
   &\sum_{j\in P}\,\widetilde{a}_{ij}=\sum_{j=1}^N\,a_{ij}\ge0\qquad\forall\,\,
   i\in R\,,\label{tsum}\\
   &\sum_{i,j=1}^M\,v_i\,\widetilde{a}_{ij}\,v_j
   \ge\sum_{i,j=1}^M\,v_i\,a_{ij}\,v_j>0\qquad
   \forall\,\,(v_1,\dots,v_M)\in{\mathbb R}^M\setminus\{0\}\,,\label{32t}\\
   &\sum_{j\in P}\,\widetilde{a}_{ij}\,u_j=g_i\qquad\forall\,\,i\in R\,.
   \label{tsyst}
\end{align}
Denote
\begin{displaymath}
   s=\max\{u_i\,;\,\,i\in P\}\,,\qquad J=\{i\in P\,;\,\,u_i=s\}\,.
\end{displaymath}
It suffices to consider the case $J\neq P$ since otherwise the validity of 
\eqref{dmp3a} and \eqref{dmp4a} is obvious. First, let us show that
\begin{equation}\label{off_diag}
   \widetilde{a}_{ij}\le0\qquad\forall\,\,i\in J\cap R,\,\,j\in P\setminus J\,.
\end{equation}
Let $i\in J\cap R$ and $j\in S_i\setminus J$. For any $k\in\mathbb N$, define 
the vector ${\rm U}^k=(u_1^k,\dots,u_N^k)$ with $u_i^k=u_i+1/k$ and $u_l^k=u_l$
for $l\neq i$. Then $u^k_i$ is a strict local maximum of ${\rm U}^k$ with 
respect to $S_i$ and hence, in view of Assumption (A2),
\begin{displaymath}
   (a_{ij}+b_{ij}({\rm U}^k))\,(u^k_i-u^k_j)\le0\,.
\end{displaymath}
Since ${\rm U}^k\to {\rm U}$ for $k\to\infty$, Assumption (A1) implies that
\begin{displaymath}
   (a_{ij}+b_{ij}({\rm U}))\,(u_i-u_j)\le0\,.
\end{displaymath}
As $u_i-u_j>0$, it follows that $a_{ij}+b_{ij}({\rm U})\le0$.
For $j\not\in S_i\cup\{i\}$, one has $a_{ij}+b_{ij}({\rm U})=0$,
which completes the proof of \eqref{off_diag}.

Now we want to prove that the relations \eqref{tsum}--\eqref{off_diag} imply
\eqref{dmp3a} and \eqref{dmp4a}. If \eqref{mp3} does not hold, it suffices to
consider $s>0$ since otherwise \eqref{dmp3a} trivially holds.
Let us assume that \eqref{dmp4a} does not hold, which
implies that $J\subset R$. We shall prove that then
\begin{equation}\label{99}
   \exists\,\,k\in J:\quad\mu_k:=\sum_{j\in J}\,\widetilde{a}_{kj}>0\,.
\end{equation}
Assume that \eqref{99} is not satisfied. Then, applying \eqref{tsum} and 
\eqref{off_diag}, one derives for any $i\in J$
\begin{displaymath}
   0\ge\sum_{j\in J}\,\widetilde{a}_{ij}
   \ge-\sum_{j\in P\setminus J}\,\widetilde{a}_{ij}\ge0\,,
\end{displaymath}
which gives
\begin{displaymath}
   \sum_{j\in J}\,\widetilde{a}_{ij}=0\qquad\forall\,\,i\in J\,.
\end{displaymath}
Thus, the matrix $(\widetilde{a}_{ij})_{i,j\in J}$ is singular, which 
contradicts \eqref{32t}. Therefore, \eqref{99} holds and hence, denoting
$r=\max\{u_i\,;\,\,i\in P\setminus J\}\,,$
one obtains using \eqref{tsyst} and \eqref{off_diag}
\begin{equation}\label{88}
   s\,\mu_k=\sum_{j\in J}\,\widetilde{a}_{kj}\,u_j
   =g_k-\sum_{j\in P\setminus J}\,\widetilde{a}_{kj}\,u_j
   \le r\,\sum_{j\in P\setminus J}\,(-\widetilde{a}_{kj})\,.
\end{equation}
If \eqref{mp3} holds, then, in view of \eqref{tsum}, the right-hand side of 
\eqref{88} equals $r\mu_k$. Hence, $s\le r$, which is a contradiction to the 
definition of $J$. If \eqref{mp3} does not hold, then it is assumed that $s>0$
and hence, in view of \eqref{off_diag}, the inequality \eqref{88} implies 
that $r>0$. Thus, in view of \eqref{tsum}, the
right-hand side of \eqref{88} is bounded by $r\mu_k$, which again implies
that $s\le r$. Therefore \eqref{dmp4a} and hence also \eqref{dmp3a} hold.

The implications \eqref{dmp3b} and \eqref{dmp4b} can be proved
analogously.
\end{proof}

\begin{remark}
Note that $P$ may contain also indices from the set $\{M+1,\dots,N\}$.
The assumption $Q\neq\emptyset$ is always satisfied if \eqref{mp3} holds since
otherwise, due to \eqref{tsum}, the matrix $(\widetilde{a}_{ij})_{i,j\in R}$
would be singular, which is not possible in view of \eqref{32t}. If $\rm U$
satisfies \eqref{23} with $g_i\le 0$ for all $i\in R$ and $\max_{i\in P}u_i>0$, 
then it was shown in the above proof that $J\not\subset R$, which again implies 
that $Q\neq\emptyset$. The same holds if $\rm U$ satisfies \eqref{23} with
$g_i\ge 0$ for all $i\in R$ and $\min_{i\in P}u_i<0$.
\end{remark}

Setting $R=\{1,\dots,M\}$ in Theorem~\ref{thm:general_DMP}, one obtains the
following global DMP.

\begin{corollary}\label{cor:global_DMP}
Let \eqref{mp2}, \eqref{32}, and \eqref{eq_b1}--\eqref{eq_b4} hold and let
Assumptions (A1) and (A2) be satisfied. 
Then any solution ${\rm U}=(u_1,\dots,u_N)\in\RR^N$ of \eqref{23} satisfies the
global DMP
\begin{eqnarray}
   &&g_i\le 0\,,\,\,\,i=1,\dots,M\qquad\Rightarrow\qquad 
   \max_{i=1,\dots,N}u_i\le\max_{i=M+1,\dots,N}\,u_i^+\,,\label{dmp5a}\\
   &&g_i\ge 0\,,\,\,\,i=1,\dots,M\qquad\Rightarrow\qquad
   \min_{i=1,\dots,N}u_i\ge\min_{i=M+1,\dots,N}\,u_i^-\,.\label{dmp5b}
\end{eqnarray}
If, in addition, the condition \eqref{mp} holds, then
\begin{eqnarray}
   &&g_i\le 0\,,\,\,\,i=1,\dots,M\qquad\Rightarrow\qquad
   \max_{i=1,\dots,N}u_i=\max_{i=M+1,\dots,N}\,u_i\,,\label{dmp6a}\\
   &&g_i\ge 0\,,\,\,\,i=1,\dots,M\qquad\Rightarrow\qquad
   \min_{i=1,\dots,N}u_i=\min_{i=M+1,\dots,N}\,u_i\,.\label{dmp6b}
\end{eqnarray}
\end{corollary}

Finally, let us return to the convection--diffusion--reaction problem
\eqref{strong-steady} and assume that the algebraic problem
\eqref{21}, \eqref{21b} is defined by \eqref{13}--\eqref{15} with $a_h$ given
by \eqref{eq_form_a} or \eqref{eq_lumped_react}. Recall that a vector
${\rm U}=(u_1,\dots,u_N)$ can be identified with a function $u_h\in W_h$ via
\eqref{eq_vect_fcn_identification}. Then, for index sets $S_i$
defined by \eqref{def_S}, Theorem~\ref{thm:general_DMP} implies that finite
element functions $u_h\in W_h$ corresponding to ${\rm U}\in \mathbb R^N$ 
obeying to \eqref{23} satisfy an analog of the continuous maximum principle
\eqref{eq_max_1a}--\eqref{eq_max_2b}.

\begin{theorem}\label{thm:general_DMP2}
Let the assumptions stated in Sect.~\ref{s1} be satisfied and let the 
algebraic problem \eqref{21}, \eqref{21b} be defined by \eqref{13}--\eqref{15} 
with $a_h$ given by \eqref{eq_form_a} or \eqref{eq_lumped_react}. Let the index
sets $S_i$ be given by \eqref{def_S}. Consider a matrix ${\mathbb B}({\rm
U})\in\mathbb R^{N\times N}$ depending on ${\rm U}\in\mathbb R^N$ and
satisfying \eqref{eq_b1}--\eqref{eq_b3}, \eqref{eq_b4}, and Assumptions (A1)
and (A2). Consider any nonempty set ${\mathscr G}_h\subset \TT_h$ and define
\begin{displaymath}
   G_h=\bigcup_{T\in {\mathscr G}_h}\,T\,.
\end{displaymath}
Let ${\rm U}\in{\mathbb R}^N$ be a solution of \eqref{23} and let $u_h\in W_h$
be the corresponding finite element function given by
\eqref{eq_vect_fcn_identification}. Then one has the DMP
\begin{align}
   &\rhs\le0\quad\mbox{\rm in}\,\,\,G_h\qquad\Rightarrow\qquad
    \max_{G_h}\,u_h\le\max_{\partial G_h}\,u_h^+\,,\label{eq_dmp_1a}\\
   &\rhs\ge0\quad\mbox{\rm in}\,\,\,G_h\qquad\Rightarrow\qquad
    \min_{G_h}\,u_h\ge\min_{\partial G_h}\,u_h^-\,.\label{eq_dmp_1b}
\end{align}
If, in addition, $c=0$ in $G_h$, then
\begin{align}
   &\rhs\le0\quad\mbox{\rm in}\,\,\,G_h\qquad\Rightarrow\qquad
    \max_{G_h}\,u_h=\max_{\partial G_h}\,u_h\,,\label{eq_dmp_2a}\\
   &\rhs\ge0\quad\mbox{\rm in}\,\,\,G_h\qquad\Rightarrow\qquad
    \min_{G_h}\,u_h=\min_{\partial G_h}\,u_h\,.\label{eq_dmp_2b}
\end{align}
\end{theorem}

\begin{proof}
Set
\begin{displaymath}
   R:=\{i\in\{1,\dots,M\}\,;\,\,x_i\in\mbox{\rm int}\,G_h\}\,,\qquad
   P^\prime:=\{i\in\{1,\dots,N\}\,;\,\,x_i\in G_h\}\,,
\end{displaymath}
where $\mbox{\rm int}\,G_h$ denotes the interior of $G_h$. Since $u_i=u_h(x_i)$
for any $i\in P^\prime$ and $u_h$ is piecewise linear, one has
\begin{equation}\label{51}
   \max_{G_h}\,u_h=\max_{i\in P^\prime}u_i\,,\qquad\quad
   \min_{G_h}\,u_h=\min_{i\in P^\prime}u_i\,.
\end{equation}
If $R=\emptyset$, then $x_i\in\partial G_h$ for any $i\in P^\prime$ and
\eqref{51} immediately implies the validity of the right-hand sides in the
implications \eqref{eq_dmp_1a}--\eqref{eq_dmp_2b}. Thus, assume that 
$R\neq\emptyset$. Let $P$ and $Q$ be defined by \eqref{eq_P_Q}. Then, in view 
of the definition of $S_i$, one has $P\subset P^\prime$ and $Q\neq\emptyset$. 
If $\rhs\le 0$ in $G_h$, then $g_i\le0$ for any $i\in R$ and hence
\begin{displaymath}
   \max_{i\in P}u_i\le\max_{i\in Q}\,u_i^+\le\max_{\partial G_h}\,u_h^+
\end{displaymath}
according to \eqref{dmp3a}. If $i\in P^\prime\setminus P$, then 
$x_i\in\partial G_h$ and hence
\begin{displaymath}
   u_i=u_h(x_i)\le\max_{\partial G_h}\,u_h\le\max_{\partial G_h}\,u_h^+\,.
\end{displaymath}
Consequently, \eqref{eq_dmp_1a} holds due to \eqref{51}. The implications
\eqref{eq_dmp_1b}--\eqref{eq_dmp_2b} follow analogously. Note that if $c=0$ in
$G_h$, then \eqref{mp3} holds since $\sum_{j=1}^N\varphi_j=1$.
\end{proof}

\begin{remark}
It might be surprising that the local DMP proved in Theorem~\ref{thm_local_DMP}
was not employed for proving the global DMP and instead a much more complicated
proof was considered in Theorem~\ref{thm:general_DMP}. However,
the global DMP cannot be obtained as a consequence of the local DMPs as the 
following example shows. Let $u_1,\dots,u_{16}$ be values at the
vertices of the triangulation depicted in Fig.~\ref{fig:local_to_global_DMP}
\begin{figure}[t]
\centerline{
\includegraphics[width=0.2\textwidth]{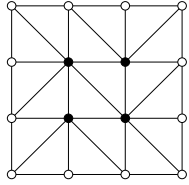}
} 
\caption{Local DMP does not imply a global DMP}\label{fig:local_to_global_DMP}
\end{figure}
numbered as in Sect.~\ref{s2}. Let $u_1=\dots=u_4=1$ (values at the black
interior vertices) and $u_5=\dots=u_{16}=0$ (values at the white boundary
vertices). Let the index sets $S_i$ be given by \eqref{def_S}. Then the local
DMP 
\begin{displaymath}
   u_i\le\max_{j\in S_i}\,u_j\,,\qquad\quad i=1,\dots,4\,,
\end{displaymath}
is satisfied but the corresponding global DMP (the right-hand sides of the 
implications \eqref{dmp5a} and \eqref{dmp6a} with $M=4$ and $N=16$) does not 
hold.
\end{remark}

\section{An error estimate}
\label{s5}

In the previous section, we analyzed the nonlinear algebraic problem
\eqref{23}, \eqref{23b} on its own, without relating it to some discretization
(except for Theorem~\ref{thm:general_DMP2}). If the algebraic problem
originates from a discretization of the convection--diffusion--reaction problem
\eqref{strong-steady}, then a natural question is how well its solution
approximates the solution $u$ of \eqref{strong-steady}. This question will be
briefly addressed in this section.

Let us assume that the algebraic problem \eqref{21}, \eqref{21b}
corresponds to the variational problem \eqref{v3} satisfying \eqref{v1}, i.e.,
it is defined by \eqref{13}--\eqref{15}.
Let $u_h\in W_h$ correspond to the solution ${\rm U}\in{\mathbb R}^N$ of the 
nonlinear algebraic problem \eqref{23}, \eqref{23b} via 
\eqref{eq_vect_fcn_identification}. Our aim is to estimate the error $u-u_h$.
To this end, it is of advantage to write the nonlinear algebraic problem in a 
variational form. We denote
\begin{displaymath}
   b_h(w;z,v)
   =\sum_{i,j=1}^N\,b_{ij}(w)\,z(x_j)\,v(x_i)
   \qquad\forall\,\,w,z,v\in C(\overline\Omega)\,,
\end{displaymath}
with $b_{ij}(w):=b_{ij}(\{w(x_i)\}_{i=1}^N)$. Then the nonlinear algebraic 
problem \eqref{23}, \eqref{23b} is equivalent to the following variational 
problem:
\vspace*{2mm}

Find $u_h\in W_h$ such that $u_h(x_i)=u_b(x_i)$, $i=M+1,\dots,N$, and
\begin{equation*}
   a_h(u_h,v_h)+b_h(u_h;u_h,v_h)=(\rhs,v_h)\qquad
   \forall\,\,v_h\in V_h\,.
\end{equation*}

In view of \eqref{eq_b1}--\eqref{eq_b3}, for any $w\in C(\overline\Omega)$, 
the mapping $b_h(w;\cdot,\cdot)$ is a nonnegative symmetric bilinear form on 
$C(\overline\Omega)\times C(\overline\Omega)$ and hence the functional 
$(b_h(w;\cdot,\cdot))^{1/2}$ is a seminorm on $C(\overline\Omega)$. Thus, for 
estimating the error $u-u_h$, it is natural to use a solution-dependent norm on
$V_h$ defined by
\begin{displaymath}
   \| v_h\| _h^{}:=\Big(C_a\,\| v_h\| _a^2+b_h(u_h;v_h,v_h)\Big)^{1/2}\,,\qquad 
   v_h\in V_h\,,
\end{displaymath}
where $C_a$ and $\| \cdot\| _a^{}$ are the same as in \eqref{v1}. Note that 
$\| \cdot\| _h^{}$ may be only a seminorm on $W_h$ and that it is not 
defined on the space $H^1(\Omega)$. Assuming that $u\in C(\overline\Omega)$ and 
using the techniques of \cite{BJK16}, one obtains the estimate
\begin{align}
   &\| u-u_h\| _h^{}\le C_a^{1/2}\,\| u-i_hu\| _a^{}
   +\sup_{v_h\in V_h}\frac{a(u,v_h)-a_h(i_hu,v_h)}{\| v_h\| _h^{}}\nonumber\\
   &\hspace*{6cm}+(b_h(u_h;i_hu,i_hu))^{1/2}\,,\label{v4}
\end{align}
where $i_h:C(\overline\Omega)\to W_h$ is the usual Lagrange interpolation
operator. The last term on the right-hand side represents an estimate of the
consistency error originating from the algebraic stabilization.

In what follows, we shall assume that either $a_h=a$ or $a_h$ is defined by
\eqref{eq_lumped_react} so that one can use the norm $\| \cdot\| _a^{}$ given by
\eqref{eq:cont_norm} and consider $C_a=1$. For simplicity, we shall assume
that $\sigma_0>0$ and refer to \cite{BJK16} for the case $\sigma_0=0$. Assuming
that $u\in H^2(\Omega)$, standard interpolation estimates (cf. \cite{Ciarlet})
give
\begin{equation}\label{72a}
  \| u-i_hu\| _a^{}\le C\,(\varepsilon+\sigma_0\,h^2)^{1/2}\,h\,\vert u\vert _{2,\Omega}^{}\,.
\end{equation}
Moreover, it was shown in \cite{BJK16} that one has
\begin{equation}
   \sup_{v_h\in V_h}\frac{a(u,v_h)-a_h(i_hu,v_h)}{\| v_h\| _h^{}}
   \le C\,(\varepsilon+\sigma_0^{-1}\,\{\| \bb\| _{0,\infty,\Omega}^2
   +\| c\| _{0,\infty,\Omega}^2\})^{1/2}\,h\,\| u\| _{2,\Omega}^{}\,.\label{34}
\end{equation}
To estimate the last term in \eqref{v4}, we assume that \eqref{eq_b4} holds
with $S_i$ defined in \eqref{def_S} for all $i=1,\dots,N$. Then it follows 
using \eqref{eq_pos_semidef} and \eqref{eq_b3} that
\begin{align*}
   b_h(u_h;i_hu,i_hu)&=-\frac12\sum_{i=1}^N\,\sum_{j\in S_i}\,b_{ij}(u_h)\,
   [u(x_i)-u(x_j)]^2\\
   &\le\sum_{T\in\TT_h}\,\sum_{x_i,x_j\in T}\,\vert b_{ij}(u_h)\vert \,[u(x_i)-u(x_j)]^2\\
   &\le\sum_{T\in\TT_h}\,\sum_{x_i,x_j\in T}\,\vert b_{ij}(u_h)\vert \,\| x_i-x_j\| ^2
   \| (\nabla i_hu)\vert _T^{}\| ^2\,,
\end{align*}
where $\| \cdot\| $ is the Euclidean norm on ${\mathbb R}^d$. Thus, using the 
shape regularity of $\TT_h$ and denoting
\begin{displaymath}
   A_h(u_h)=\max_{i,j=1,\dots,N,\,i\neq j}\,\left(\vert b_{ij}(u_h)\vert \,\| x_i-x_j\| ^{2-d}\right),
\end{displaymath}
one has
\begin{displaymath}
   b_h(u_h;i_hu,i_hu)\le C\,A_h(u_h)\,\vert i_hu\vert _{1,\Omega}^2\,.
\end{displaymath}
The behavior of $A_h(u_h)$ with respect to $h$ depends on how the artificial
diffusion matrix is constructed. Often (e.g., in the next two sections), one 
has 
\begin{equation}\label{eq:crude_est}
   \vert b_{ij}(u_h)\vert \le\max\{\vert a_{ij}\vert ,\vert a_{ji}\vert \}\qquad\forall\,\,i\neq j\,.
\end{equation}
Then (cf.~the proofs of \cite[Lemma 16]{BJK16} and \cite[Lemma 2]{BJK18})
\begin{align*}
   \vert b_{ij}(u_h)\vert \le C\,(\varepsilon+\| \bb\| _{0,\infty,\Omega}^{}\,h
   +\| c\| _{0,\infty,\Omega}^{}\,h^2)\,\| x_i-x_j\| ^{d-2}\qquad
   \forall\,\,i\neq j\,,
\end{align*}
and hence
\begin{equation}\label{est-l3}
   b_h(u_h;i_hu,i_hu)
   \le C\,(\varepsilon+\| \bb\| _{0,\infty,\Omega}^{}\,h
   +\| c\| _{0,\infty,\Omega}^{}\,h^2)\,\vert i_hu\vert _{1,\Omega}^2\,.
\end{equation}
Finally, substituting the estimates \eqref{72a}, \eqref{34}, and \eqref{est-l3}
in \eqref{v4}, one obtains the estimate
\begin{align}
   \| u-u_h\| _h^{}&\le 
   C\,(\varepsilon+\sigma_0^{-1}\,\{\| \bb\| _{0,\infty,\Omega}^2
   +\| c\| _{0,\infty,\Omega}^2\}+\sigma_0h^2)^{1/2}\,h\,\| u\| _{2,\Omega}
   \nonumber\\
   &\hspace*{16mm}+ C\,(\varepsilon+\| \bb\| _{0,\infty,\Omega}\,h
   +\| c\| _{0,\infty,\Omega}^{}\,h^2)^{1/2}\,\vert i_hu\vert _{1,\Omega}\,.
   \label{eq:error_est}
\end{align}
Note that, in all the above estimates, the constant $C$ is independent of $h$
and the data of problem \eqref{strong-steady}.

As one can see, the estimate \eqref{eq:error_est} implies the convergence order
$1/2$ in the convection-dominated case and no convergence in the 
diffusion-dominated case. It was demonstrated in \cite{BJK16} that this result 
is sharp under the above assumptions on the artificial diffusion matrix. 
However, for particular definitions of $b_{ij}$ and/or particular types of
triangulations, a better convergence behavior can be observed numerically and
in a few special cases also proved. We refer to \cite{BJK16}, \cite{BJK17}, 
and \cite{BJK18} for a refined analysis and various numerical results.

\section{Algebraic flux correction}
\label{s6}

In this section we present an example of the nonlinear algebraic problem
\eqref{23}, \eqref{23b} based on algebraic flux correction (AFC). 

A detailed derivation of an AFC scheme for problem \eqref{21}, \eqref{21b} can 
be found, e.g., in \cite{BJK16}. The idea is to add the term 
$({\mathbb D}\,{\rm U})_i$ to both sides of \eqref{21} (so that, on
the left-hand side, one has the same matrix as in the stabilized problem
\eqref{22}) and then, on the right-hand side, to use the identity
\begin{displaymath}
   ({\mathbb D}\,{\rm U})_i=\sum_{j=1}^N\,f_{ij}\qquad\mbox{with}\qquad
   f_{ij}=d_{ij}\,(u_j-u_i)
\end{displaymath}
and to limit those anti-diffusive fluxes $f_{ij}$ that would 
otherwise cause spurious oscillations. The limiting is achieved by multiplying
the fluxes by solution dependent limiters $\alpha_{ij}\in[0,1]$. This leads to 
the nonlinear algebraic problem
\begin{align}
   &\sum_{j=1}^N\,a_{ij}\,u_j
   +\sum_{j=1}^N\,(1-\alpha_{ij}({\rm U}))\,d_{ij}\,(u_j-u_i)
   =\rhs_i\,,\qquad i=1,\dots,M\,,\label{8}\\
   &u_i=u^b_i\,,\qquad i=M+1,\dots,N\,.\label{9}
\end{align}
It is assumed that
\begin{equation}
   \alpha_{ij}=\alpha_{ji}\,,\qquad i,j=1,\dots,N\,,\label{31}
\end{equation}
and that, for any $i,j\in\{1,\dots,N\}$, the function 
$\alpha_{ij}({\rm U})(u_j-u_i)$ is a continuous function of 
${\rm U}\in{\mathbb R}^N$. A theoretical analysis of the AFC scheme
\eqref{8}, \eqref{9} concerning the solvability, local DMP and
error estimation can be found in \cite{BJK16}.

The symmetry condition \eqref{31} is particularly important for several 
reasons. First, it guarantees that the resulting method is conservative.
Second, it implies that the matrix corresponding to the term arising from the 
AFC is positive semidefinite. This shows that this term really enhances the
stability of the method and enables to estimate the error of the approximate
solution, see \cite{BJK16}. Finally, it was demonstrated in \cite{BJK15} that, 
without the symmetry condition \eqref{31}, the nonlinear algebraic problem 
\eqref{8}, \eqref{9} is not solvable in general.

In view of the equivalence between \eqref{23} and \eqref{23c}, it is obvious
that \eqref{8} can be written in the form \eqref{23} with
\begin{equation}\label{afc-bij}
   b_{ij}({\rm U})=(1-\alpha_{ij}({\rm U}))\,d_{ij}\qquad\forall\,\,i\neq j\,,
   \qquad\qquad
   b_{ii}({\rm U})=-\sum_{j\neq i}\,b_{ij}({\rm U})\,.
\end{equation}
This matrix $(b_{ij}({\rm U}))_{i,j=1}^N$ satisfies the assumptions
\eqref{eq_b1}--\eqref{eq_b3} and \eqref{eq_b4} with $S_i$ defined by
\eqref{ass_S_1}.

Of course, the properties of the AFC scheme \eqref{8}, \eqref{9} significantly
depend on the choice of the limiters $\alpha_{ij}$. Here we present the Kuzmin
limiter 
proposed in \cite{Kuzmin07} which was thoroughly investigated in \cite{BJK16}
and can be considered as a standard limiter for algebraic stabilizations of 
steady-state convection--diffusion--reaction equations.

To define the limiter of \cite{Kuzmin07}, one first computes, for
$i=1,\dots,M$,
\begin{equation}\label{46}
   P_i^+=\sum_{\mbox{\parbox{8mm}{\scriptsize\centerline{$j{=}1$}
   \centerline{$a_{ji}\le a_{ij}$}}}}^N\,f_{ij}^+\,,\quad\,\,\,
   P_i^-=\sum_{\mbox{\parbox{8mm}{\scriptsize\centerline{$j{=}1$}
   \centerline{$a_{ji}\le a_{ij}$}}}}^N\,f_{ij}^-\,,\quad\,\,\,
   Q_i^+=-\sum_{j=1}^N\,f_{ij}^-\,,\quad\,\,\,
   Q_i^-=-\sum_{j=1}^N\,f_{ij}^+\,,
\end{equation}
where $f_{ij}=d_{ij}\,(u_j-u_i)$,
$f_{ij}^+=\max\{0,f_{ij}\}$, and $f_{ij}^-=\min\{0,f_{ij}\}$. Then, one defines
\begin{equation}\label{R-definition}
   R_i^+=\min\left\{1,\frac{Q_i^+}{P_i^+}\right\},\quad
   R_i^-=\min\left\{1,\frac{Q_i^-}{P_i^-}\right\},\qquad
   i=1,\dots,M\,.
\end{equation}
If $P_i^+$ or $P_i^-$ vanishes, one sets $R_i^+=1$ or $R_i^-=1$, respectively.
For $i=M+1,\dots,N$, one defines $R_i^+=R_i^-=1$. Furthermore, one sets
\begin{equation}\label{alpha-definition}
        \widetilde\alpha_{ij}=\left\{
        \begin{array}{cl}
                R_i^+\quad&\mbox{if}\,\,\,f_{ij}>0\,,\\
                1\quad&\mbox{if}\,\,\,f_{ij}=0\,,\\
                R_i^-\quad&\mbox{if}\,\,\,f_{ij}<0\,,
        \end{array}\right.\qquad\qquad
        i,j=1,\dots,N\,.
\end{equation}
Finally, one defines
\begin{equation}\label{symm_alpha}
   \alpha_{ij}=\alpha_{ji}=\widetilde\alpha_{ij}\qquad\mbox{if}\quad
   a_{ji}\le a_{ij}\,,\qquad i,j=1,\dots,N\,.
\end{equation}

It was proved in \cite{BJK16} that the AFC scheme \eqref{8}, \eqref{9} with
the above limiter satisfies the local DMP \eqref{dmp1} with $S_i$ defined by
\eqref{def-S1} provided that
\begin{equation}\label{a2}
   a_{ij}+a_{ji}\le0\qquad\forall\,\,
   i,j=1,\dots,N\,,\,\, i\neq j\,,\,\,i\le M\,\,\mbox{or}\,\,j\le M\,.
\end{equation}
The local DMP \eqref{dmp2} holds under the additional condition \eqref{mp}.
In \cite{Kno17}, it was proved that the assumption \eqref{a2} can be weakened 
to
\begin{equation}\label{assumption_min}
   \min\{a_{ij},a_{ji}\}\le0\qquad
   \forall\,\,i=1,\dots,M\,,\,\,j=1,\dots,N\,,\,\,i\neq j\,.
\end{equation}
Then the local DMP \eqref{dmp1} holds with $S_i$ defined by \eqref{ass_S_1}
and, if \eqref{mp} is satisfied, then again also the local DMP \eqref{dmp2} is 
valid.

If the AFC scheme \eqref{8}, \eqref{9} is applied to the algebraic problem
\eqref{21}, \eqref{21b} defined by \eqref{13}--\eqref{15} with $a_h$ given by 
\eqref{eq_lumped_react}, then, as discussed in \cite{BJK16}, the validity of 
\eqref{a2} is guaranteed if the triangulation $\TT_h$ is weakly acute, i.e., 
if the angles between facets of $\TT_h$ do not exceed $\pi/2$. In the 
two-dimensional case, \eqref{a2} holds if and (in principle) only if $\TT_h$ 
is a Delaunay triangulation, i.e., the sum of any pair of angles opposite a
common edge is smaller than, or equal to, $\pi$ (the note `in principle' is
added because angles opposite interior edges having both end points on the
boundary of $\Omega$ can be arbitrary). The condition 
\eqref{assumption_min} may be satisfied also for non-Delaunay triangulations,
particularly, in the convection-dominated case, since the convection matrix is
skew-symmetric. However, in general, the validity of a DMP cannot be guaranteed 
for non-Delaunay triangulations. Moreover, if the lumped bilinear form 
\eqref{eq_lumped_react} is replaced
by the original bilinear form \eqref{eq_form_a}, then the validity of the 
conditions \eqref{a2} or \eqref{assumption_min} may be lost since some
off-diagonal entries of the matrix corresponding to the reaction term from 
\eqref{eq_form_a} are positive.

It was shown in \cite{Kno17} that the DMP generally does not hold if 
condition \eqref{assumption_min} is not satisfied. This is due to the condition 
$a_{ji}\le a_{ij}$ used in \eqref{symm_alpha} to symmetrize the factors 
$\widetilde\alpha_{ij}$. It suffices to study this condition for $i\le M$ or 
$j\le M$ since $\alpha_{ij}$ with $i,j\in\{M+1,\dots,N\}$ does not occur in 
\eqref{8}. Then, if the discretizations from Sect.~\ref{s2} are considered, 
the symmetry of the bilinear forms corresponding to the diffusion and reaction
terms implies that the condition $a_{ji}<a_{ij}$ is equivalent to the 
inequality
\begin{equation*}
   ({\bb}\cdot\nabla\varphi_j,\varphi_i)>0\,.
\end{equation*}
As it was discussed in \cite{Kno17}, in many cases (depending on $\bb$ and the
geometry of the triangulation), this inequality means that the vertex $x_i$
lies in the upwind direction with respect to the vertex $x_j$. Consequently,
the use of the inequality $a_{ji}<a_{ij}$ in the definition of the above 
limiter causes that $\alpha_{ij}=\alpha_{ji}$ is defined using quantities 
computed at the upwind vertex of the edge with end points $x_i$, $x_j$. 
It turns out that this feature has a positive 
influence on the quality of the approximate solutions and on the convergence of 
the iterative process for solving the nonlinear problem \eqref{8}, \eqref{9}. 

In order to obtain a method satisfying the DMP on arbitrary meshes and
preserving the upwind feature described above, modifications of
$\alpha_{ij}=\alpha_{ji}$ were considered in \cite{Kno17,Kno19} if 
$\min\{a_{ij},a_{ji}\}>0$. In the present paper, we shall achieve this goal by
changing the definition of the matrix ${\mathbb B}({\rm U})$ in 
\eqref{afc-bij}. First, however, we shall derive an equivalent form of the
above limiter under the assumption \eqref{assumption_min}. Note that, without 
this assumption, the application of the limiter does not make much sense since 
the main goal of the AFC, i.e., the validity of the DMP, is not achieved in
general. Moreover, if \eqref{assumption_min} does not hold, the AFC scheme is 
not uniquely defined because the symmetrization \eqref{symm_alpha} is ambiguous 
if $a_{ij}=a_{ji}$. If \eqref{assumption_min} holds, this ambiguity does not
influence the resulting method since $d_{ij}=0$ for $a_{ij}=a_{ji}$ and hence
the respective $\alpha_{ij}=\alpha_{ji}$ does not occur in the nonlinear 
problem \eqref{8}, \eqref{9} and can be defined arbitrarily.

Thus, let us assume that \eqref{assumption_min} holds. Then, for any
$i\in\{1,\dots,M\}$ and $j\in\{1,\dots,N\}$ with $i\neq j$, one has the 
equivalence
\begin{equation*}
   a_{ji}\le a_{ij}\quad\mbox{and}\quad d_{ij}\neq0\qquad\Leftrightarrow\qquad
   a_{ij}>0\,.
\end{equation*}
Moreover, if $a_{ij}>0$, then $d_{ij}=-a_{ij}$. Therefore, it follows from 
\eqref{46} that
\begin{equation}\label{46a}
   P_i^+=\sum_{\mbox{\parbox{8mm}{\scriptsize\centerline{$j{=}1$}
   \centerline{$a_{ij}>0$}}}}^N\,a_{ij}\,(u_i-u_j)^+\,,\qquad\quad
   P_i^-=\sum_{\mbox{\parbox{8mm}{\scriptsize\centerline{$j{=}1$}
   \centerline{$a_{ij}>0$}}}}^N\,a_{ij}\,(u_i-u_j)^-\,.
\end{equation}
Furthermore, we shall rewrite the formulas for $Q_i^\pm$ and
$\widetilde\alpha_{ij}$. For this, the validity of \eqref{assumption_min} will
not be needed. Since, for any real number $a$, its positive and negative parts
satisfy $-a^-=(-a)^+$ and $-a^+=(-a)^-$, one has
\begin{equation}\label{46b}
   Q_i^+=\sum_{j=1}^N\,\vert d_{ij}\vert \,(u_j-u_i)^+\,,\qquad\quad
   Q_i^-=\sum_{j=1}^N\,\vert d_{ij}\vert \,(u_j-u_i)^-\,.
\end{equation}
If $d_{ij}\neq0$, then
\begin{equation}\label{alpha-definition_new}
        \widetilde\alpha_{ij}=\left\{
        \begin{array}{cl}
                R_i^+\quad&\mbox{if}\,\,\,u_i>u_j\,,\\
                1\quad&\mbox{if}\,\,\,u_i=u_j\,,\\
                R_i^-\quad&\mbox{if}\,\,\,u_i<u_j\,.
        \end{array}\right.\qquad\qquad
\end{equation}
If $d_{ij}=0$, then \eqref{alpha-definition_new} generally gives another value
than \eqref{alpha-definition} but since $\alpha_{ij}$ is multiplied by
$d_{ij}$ in \eqref{8}, the use of \eqref{alpha-definition_new} does not
change the AFC scheme. Thus, if the condition \eqref{assumption_min} is
satisfied, then defining the limiter $\alpha_{ij}$ in the AFC scheme
\eqref{8}, \eqref{9} by \eqref{46a}, \eqref{46b}, \eqref{R-definition},
\eqref{alpha-definition_new}, and \eqref{symm_alpha} is equivalent to using 
\eqref{46}--\eqref{symm_alpha}.

\section{A new algebraically stabilized scheme}
\label{s7}

As discussed in the preceding section, the symmetrization \eqref{symm_alpha}
of the limiter causes that the DMP does not hold for the AFC scheme
\eqref{8}, \eqref{9} in general. In this section we modify the AFC scheme in
such a way that the symmetry of the limiter will not be needed and the DMP will
be always satisfied.

To make the formulas clearer, we denote
\begin{equation}\label{beta-alpha}
   \beta_{ij}=1-\alpha_{ij}\,.
\end{equation}
As we know, the AFC scheme \eqref{8}, \eqref{9} can be written in the form 
\eqref{23}, \eqref{23b} with the artificial diffusion matrix 
${\mathbb B}({\rm U})=(b_{ij}({\rm U}))_{i,j=1}^N$ given in \eqref{afc-bij}. 
In view of \eqref{def-dij} and \eqref{31}, one observes that the off-diagonal 
entries of this matrix satisfy
\begin{equation*}
   b_{ij}({\rm U})=-\beta_{ij}({\rm U})\max\{a_{ij},0,a_{ji}\}
   =-\max\{\beta_{ij}({\rm U})\,a_{ij},0,\beta_{ji}({\rm U})\,a_{ji}\}\,.
\end{equation*}
This motivates us to define the artificial diffusion matrix by
\begin{align}
   b_{ij}({\rm U})
   &=-\max\{\beta_{ij}({\rm U})\,a_{ij},0,\beta_{ji}({\rm U})\,a_{ji}\}\,,\qquad
   i,j=1,\dots,N\,,\,\,i\neq j\,,\label{asm-bij}\\
   b_{ii}({\rm U})
   &=-\sum_{\mbox{\parbox{4mm}{\scriptsize\centerline{$j{=}1$}
   \centerline{$j{\neq}i$}}}}^N\,b_{ij}({\rm U})\,,\qquad i=1,\dots,N\,.
   \label{asm-bii}
\end{align}
Obviously, this matrix $(b_{ij}({\rm U}))_{i,j=1}^N$ again satisfies the 
assumptions \eqref{eq_b1}--\eqref{eq_b3} and \eqref{eq_b4} with $S_i$ defined 
by \eqref{ass_S_1}. Note however that, in contrast to \eqref{afc-bij}, the 
formula \eqref{asm-bij} leads to a symmetric matrix ${\mathbb B}({\rm U})$ 
also if the limiters $\alpha_{ij}$ are not symmetric. This enables us to get 
rid of the symmetry condition \eqref{31}.

Thus, we shall consider the algebraic problem \eqref{23}, \eqref{23b} with the
artificial diffusion matrix given by \eqref{asm-bij} and \eqref{asm-bii} and
with any functions $\beta_{ij}$ satisfying, for any $i,j\in\{1,\dots,N\}$,
\begin{align}
   &\beta_{ij}\,:\,{\mathbb R}^N\to[0,1]\,,\label{asm_beta1}\\
   &\mbox{if $a_{ij}>0$, then $\beta_{ij}({\rm U})(u_j-u_i)$ is a continuous 
          function of ${\rm U}\in{\mathbb R}^N$}\,.\label{asm_beta2}
\end{align}
No other assumptions on $\beta_{ij}$ will be made in the general case.

First let us state an existence result.

\begin{theorem}\label{asm_existence}
Let \eqref{32} hold and let the matrix $(b_{ij}({\rm U}))_{i,j=1}^N$ be 
defined by \eqref{asm-bij} and \eqref{asm-bii} with functions $\beta_{ij}$ 
satisfying \eqref{asm_beta1} and \eqref{asm_beta2} for any 
$i,j\in\{1,\dots,N\}$. Then Assumption~(A1) is satisfied and the nonlinear 
algebraic problem \eqref{23}, \eqref{23b} has a solution.
\end{theorem}

\begin{proof}
In view of Theorem \ref{existence}, it suffices to verify the validity of
Assumption~(A1). Consider any $i,j\in\{1,\dots,N\}$ with $i\neq j$. Due to 
\eqref{asm_beta1}, it is obvious that $b_{ij}({\rm U})$ is bounded on 
${\mathbb R}^N$ and it remains to show the continuity of 
$\Phi({\rm U}):=b_{ij}({\rm U})(u_j-u_i)$. Due to the definition of
$b_{ij}({\rm U})$, this is particularly easy if $a_{ij}\le0$ or $a_{ji}\le0$
since $\Phi({\rm U})\equiv0$ if both $a_{ij}$ and $a_{ji}$ are nonpositive and
otherwise the continuity of $\Phi({\rm U})$ immediately follows from
\eqref{asm_beta2}. Thus, let $a_{ij}>0$
and $a_{ji}>0$. Choose any $\bar{\rm U}=(\bar{u}_1,\dots,\bar{u}_N)\in\RR^N$ 
and let us show that $\Phi$ is continuous at the point $\bar{\rm U}$. If 
$\bar{u}_i=\bar{u}_j$, then $\Phi(\bar{\rm U})=0$ and the continuity at 
$\bar{\rm U}$ follows from the estimates
\begin{equation}\label{ui_is_uj}
   \vert \Phi({\rm U})-\Phi(\bar{\rm U})\vert =\vert \Phi({\rm U})\vert 
   \le C\,\vert u_i-u_j\vert \le C\,\sqrt2\,\| {\rm U}-\bar{\rm U}\| \,,
\end{equation}
where $\| \cdot\| $ is the Euclidean norm on ${\mathbb R}^N$. Thus, let
$\bar{u}_i\neq\bar{u}_j$. Without loss of generality, one can assume that
$\bar{u}_i>\bar{u}_j$. Then, if ${\rm U}\in {\mathbb R}^N$ satisfies 
$\| {\rm U}-\bar{\rm U}\| \le\frac12\vert \bar{u}_i-\bar{u}_j\vert $, one has $u_i>u_j$
and hence 
\begin{equation*}
   \Phi({\rm U})=\max\{\beta_{ij}({\rm U})\,(u_i-u_j)\,a_{ij},
                       \beta_{ji}({\rm U})\,(u_i-u_j)\,a_{ji}\}\,.
\end{equation*}
Since the maximum of two continuous functions is continuous, it follows from 
\eqref{asm_beta2} that $\Phi$ is continuous in a neighborhood of 
$\bar{\rm U}$, which completes the proof.
\end{proof}

If the functions $\beta_{ij}$ form a symmetric matrix and $\alpha_{ij}$ satisfy
\eqref{beta-alpha}, then the matrix ${\mathbb B}({\rm U})$ defined by
\eqref{asm-bij}, \eqref{asm-bii} satisfies \eqref{afc-bij} and method
\eqref{23}, \eqref{23b} can be written in the form \eqref{8}, \eqref{9}. Hence,
in this case, the AFC scheme is recovered.

Another interesting observation can be made if condition 
\eqref{assumption_min} is satisfied. Consider any $i\in\{1,\dots,M\}$ and 
$j\in\{1,\dots,N\}$ with $i\neq j$. Then, if $a_{ij}>0$, one has $a_{ji}\le0$
and hence $b_{ij}({\rm U})=-\beta_{ij}({\rm U})\,a_{ij}
=\beta_{ij}({\rm U})\,d_{ij}$. Similarly, if $a_{ji}>0$, then $a_{ij}\le0$
and hence $b_{ij}({\rm U})=-\beta_{ji}({\rm U})\,a_{ji}
=\beta_{ji}({\rm U})\,d_{ij}$. If both $a_{ij}\le0$ and $a_{ji}\le0$, then
$b_{ij}({\rm U})=0$ and $d_{ij}=0$. Thus, one concludes that
\begin{equation*}
        b_{ij}({\rm U})=\left\{
        \begin{array}{cl}
         \beta_{ij}({\rm U})\,d_{ij}\quad&\mbox{if}\,\,\,a_{ji}\le a_{ij}\,,\\
         \beta_{ji}({\rm U})\,d_{ij}\quad&\mbox{otherwise}\,,
        \end{array}\right.
\end{equation*}
for $i=1,\dots,M$ and $j=1,\dots,N$ with $i\neq j$.
Thus, if \eqref{assumption_min} holds, then the definition \eqref{asm-bij}
implicitly comprises the favorable upwind feature discussed in the preceding
section and the method \eqref{23}, \eqref{23b} can be again written in the form 
of the AFC scheme \eqref{8}, \eqref{9}. Moreover, if one sets
\begin{equation}\label{beta-talpha}
   \beta_{ij}=1-\widetilde\alpha_{ij}\,,
\end{equation}
then one obtains the AFC scheme \eqref{8}, \eqref{9} with limiters
$\alpha_{ij}$ defined by \eqref{symm_alpha}. Consequently, if the condition
\eqref{assumption_min} holds, then the AFC scheme \eqref{8}, \eqref{9} with 
limiters $\alpha_{ij}$ defined by \eqref{46}--\eqref{symm_alpha}
is equivalent to the system \eqref{23}, \eqref{23b} with ${\mathbb B}({\rm U})$ 
defined by \eqref{asm-bij}, \eqref{asm-bii}, and \eqref{beta-talpha} with 
$\widetilde\alpha_{ij}$ given by \eqref{46a}, \eqref{46b}, 
\eqref{R-definition}, and \eqref{alpha-definition_new}. Therefore, this new 
method preserves the advantages of the AFC scheme from the preceding
section which are available under condition \eqref{assumption_min}.
However, in contrast to the method from the preceding section, we shall see
that the new method satisfies the DMP also if condition
\eqref{assumption_min} is not satisfied.

For the convenience of the reader, we first summarize the definition of
$\beta_{ij}$ in the new method. We shall make a slight change in \eqref{46b}
and replace $\vert d_{ij}\vert=\max\{a_{ij},0,a_{ji}\}$ by
\begin{equation}\label{eq:qij}
   q_{ij}=\max\{\vert a_{ij}\vert,a_{ji}\}\,,
\end{equation}
which is larger or equal to $\vert d_{ij}\vert$. This heuristic modification
may improve the accuracy and convergence behavior in the
diffusion-dominated case when the method is applied to the discretizations from
Sect.~\ref{s2} and non-Delaunay meshes are used, see the discussion
in Sect.~\ref{numerics}. One could also consider the symmetric
variant $\max\{\vert a_{ij}\vert,\vert a_{ji}\vert\}$ which often leads to 
very similar results as \eqref{eq:qij}, however, in a few cases, we observed
that \eqref{eq:qij} is more convenient from the point of view of both the 
quality of the solution and the convergence of the solver used to solve the
nonlinear discrete problem.
Thus, the final definition of $\beta_{ij}$ is as 
follows. For any $i\in\{1,\dots,M\}$, set
\begin{alignat}{2}
   &P_i^+=\sum_{\mbox{\parbox{8mm}{\scriptsize\centerline{$j{=}1$}
   \centerline{$a_{ij}>0$}}}}^N\,a_{ij}\,(u_i-u_j)^+\,,\qquad\quad
   &&P_i^-=\sum_{\mbox{\parbox{8mm}{\scriptsize\centerline{$j{=}1$}
   \centerline{$a_{ij}>0$}}}}^N\,a_{ij}\,(u_i-u_j)^-\,,\label{def-beta_ij_1}
   \\[1mm]
   &Q_i^+=\sum_{j=1}^N\,q_{ij}\,(u_j-u_i)^+\,,\qquad\quad
   &&Q_i^-=\sum_{j=1}^N\,q_{ij}\,(u_j-u_i)^-\,,\label{def-beta_ij_2}\\[2mm]
   &R_i^+=\min\left\{1,\frac{Q_i^+}{P_i^+}\right\}\,,
   &&R_i^-=\min\left\{1,\frac{Q_i^-}{P_i^-}\right\}\,,\label{def-beta_ij_3}
\end{alignat}
where $q_{ij}$ is defined by \eqref{eq:qij}. Furthermore, set
\begin{equation}\label{def-beta_ij_4}
   R_i^+=1\,,\qquad R_i^-=1\,,\qquad\quad i=M+1,\dots,N\,.
\end{equation}
Then define
\begin{equation}\label{def-beta_ij_5}
        \beta_{ij}=\left\{
        \begin{array}{ll}
                1-R_i^+\quad&\mbox{if}\,\,\,u_i>u_j\,,\\
                0\quad&\mbox{if}\,\,\,u_i=u_j\,,\\
                1-R_i^-\quad&\mbox{if}\,\,\,u_i<u_j\,,
        \end{array}\right.\qquad\qquad i,j=1,\dots,N\,.
\end{equation}

\begin{remark}
\label{remark_betaij}
If $P_i^+=0$, then $R_i^+$ can be defined arbitrarily (and the same holds for 
$P_i^-$ and $R_i^-$). Indeed, $P_i^+$ is used only for defining $\beta_{ij}$
with $j$ such that $u_i>u_j$. Then, if $P_i^+=0$, one has $a_{ij}\le0$ and
hence the matrix ${\mathbb B}({\rm U})$ defined by \eqref{asm-bij},
\eqref{asm-bii} does not depend on these $\beta_{ij}$.
\end{remark}

In view of Theorem~\ref{asm_existence}, the following lemma implies that the 
problem \eqref{23}, \eqref{23b} with the artificial diffusion matrix defined by
\eqref{asm-bij}, \eqref{asm-bii} and 
\eqref{def-beta_ij_1}--\eqref{def-beta_ij_5} is solvable.

\begin{lemma}\label{asm_beta_cont}
The functions $\beta_{ij}$ defined by
\eqref{def-beta_ij_1}--\eqref{def-beta_ij_5} satisfy the assumption 
\eqref{asm_beta2} for all $i,j\in\{1,\dots,N\}$.
\end{lemma}

\begin{proof}
Consider any $i,j\in\{1,\dots,N\}$ such that $i\neq j$ and $a_{ij}>0$ and any 
$\bar{\rm U}=(\bar{u}_1,\dots,\bar{u}_N)\in\RR^N$. Like in the proof of
Theorem~\ref{asm_existence}, we want to show that 
$\Phi({\rm U}):=\beta_{ij}({\rm U})(u_j-u_i)$ is continuous at the point
$\bar{\rm U}$. If $\bar{u}_i=\bar{u}_j$, the continuity follows again from
\eqref{ui_is_uj}. If $\bar{u}_i>\bar{u}_j$, one again uses the fact that 
$u_i>u_j$ for ${\rm U}$ in a ball $B$ around $\bar{\rm U}$. Thus, for 
${\rm U}\in B$, one has
\begin{equation*}
   \Phi({\rm U})=(1-R^+_i({\rm U}))\,(u_j-u_i).
\end{equation*}
Since both $P_i^+$ and $Q_i^+$ are continuous and $P_i^+$ is positive in
$B$, the function $\Phi$ is continuous in $B$ and hence also at $\bar{\rm U}$.
If $\bar{u}_i<\bar{u}_j$, one proceeds analogously.
\end{proof}

\begin{remark}
\label{remark_continuity}
It is easy to show that $\beta_{ij}({\rm U})=\beta_{ij}(\alpha\,{\rm U})$ for
any ${\rm U}\in{\mathbb R}^N$ and any $\alpha\neq0$. This implies that 
$\beta_{ij}$ itself is not continuous since otherwise one would conclude that 
$\beta_{ij}({\rm U})=0$ for any ${\rm U}\in{\mathbb R}^N$ due to the fact that
$\beta_{ij}(0)=0$.
\end{remark}

Now let us investigate the validity of Assumption (A2).

\begin{theorem}\label{A2_for_ASM}
Let the matrix $(b_{ij}({\rm U}))_{i,j=1}^N$ be defined by \eqref{asm-bij},
\eqref{asm-bii} and \eqref{def-beta_ij_1}--\eqref{def-beta_ij_5}. Then
Assumption (A2) holds with $S_i$ defined in \eqref{ass_S_1}.
\end{theorem}

\begin{proof}
Consider any ${\rm U}=(u_1,\dots,u_N)\in\RR^N$, $i\in\{1,\dots,M\}$, and 
$j\in S_i$. Let $u_i$ be a strict local extremum of $\rm U$ with respect to 
$S_i$. We want to prove that
\begin{equation}\label{aa}
   a_{ij}+b_{ij}({\rm U})\le0\,.
\end{equation}
If $a_{ij}\le0$, then \eqref{aa} holds since $b_{ij}({\rm U})\le0$. Thus, let
$a_{ij}>0$. If $u_i>u_k$ for any $k\in S_i$, then 
$P_i^+\ge a_{ij}\,(u_i-u_j)^+>0$, $Q_i^+=0$ and hence $\beta_{ij}=1-R_i^+=1$.
Similarly, if $u_i<u_k$ for any $k\in S_i$, then 
$P_i^-\le a_{ij}\,(u_i-u_j)^-<0$, $Q_i^-=0$ and hence $\beta_{ij}=1-R_i^-=1$. 
Thus, $b_{ij}({\rm U})\le -a_{ij}$, which proves \eqref{aa}.
\end{proof}

Theorems~\ref{asm_existence} and \ref{A2_for_ASM} show that, assuming the
validity of \eqref{mp2} and \eqref{32}, solutions of the nonlinear algebraic 
problem \eqref{23}, \eqref{23b} with the artificial diffusion matrix defined 
by \eqref{asm-bij}, \eqref{asm-bii} and 
\eqref{def-beta_ij_1}--\eqref{def-beta_ij_5} satisfy all the versions of the
DMP formulated in Theorems~\ref{thm_local_DMP} and~\ref{thm:general_DMP} and
Corollary~\ref{cor:global_DMP}, without any additional assumptions on the
matrix $\mathbb A$. Therefore, if this new method is applied to the algebraic
problem \eqref{21}, \eqref{21b} defined by \eqref{13}--\eqref{15}, the DMPs
hold for both definitions \eqref{eq_form_a} and \eqref{eq_lumped_react} of the
bilinear form and for any triangulation $\TT_h$. Moreover, since $b_{ij}$
defined by \eqref{asm-bij} satisfies \eqref{eq:crude_est}, the finite element
function $u_h$ corresponding to the solution of \eqref{23}, \eqref{23b}
satisfies the error estimate \eqref{eq:error_est}.

\begin{remark}
If \eqref{def-beta_ij_1} is replaced by the original definition of $P_i^\pm$ 
from \eqref{46}, then the algebraically stabilized scheme introduced in this
section is not well defined. Indeed, in this case, $P_i^\pm$ may vanish also if
$a_{ij}>0$ so that the corresponding $\beta_{ij}$ (which may be not well 
defined) is needed for computing the matrix ${\mathbb B}({\rm U})$ defined by
\eqref{asm-bij}, \eqref{asm-bii} (cf.~also Remark~\ref{remark_betaij}).
Moreover, one can show that, independently of how 
$R_i^\pm$ are defined in these cases, the continuity assumption 
\eqref{asm_beta2} is not satisfied in general.
\end{remark}

\begin{remark}
As we already mentioned, a special case of the nonlinear algebraic problem 
\eqref{23}, \eqref{23b} with the artificial diffusion matrix defined
by \eqref{asm-bij} and \eqref{asm-bii} is the AFC scheme from Sect.~\ref{s6}.
Another example of a method having this structure is the nonlinear
stabilization based on a graph-theoretic approach described in \cite{BB17}.
Here, the artificial diffusion matrix ${\mathbb B}({\rm U})$ is given by
\begin{displaymath}
   ({\mathbb B}({\rm U})\,{\rm V})_i
   =\sum_{j\in S_i\cup\{i\}}\,\nu_{ij}({\rm U})\,l_{ij}\,v_j\qquad
   \forall\,\,{\rm V}\in{\mathbb R}^N,\,i=1,\dots,N\,,
\end{displaymath}
where $S_i$ is defined by \eqref{def_S}, $l_{ij}:=2\,\delta_{ij}-1$ is the
graph-theoretic Laplacian, and $\nu_{ij}$ is the artificial diffusion given by
\begin{displaymath}
   \nu_{ij}({\rm U})
   =\max\{\alpha_i({\rm U})\,a_{ij},0,\alpha_j({\rm U})\,a_{ji}\}\quad
   \forall\,\,i\neq j\,,\qquad
   \nu_{ii}({\rm U})=\sum_{j\in S_i}\,\nu_{ij}({\rm U})\,,
   \qquad
\end{displaymath}
with a shock detector $\alpha_i({\rm U})\in[0,1]$. Thus, the artificial
diffusion matrix satisfies \eqref{asm-bij} and \eqref{asm-bii} with 
$\beta_{ij}=\alpha_i$ for $i,j=1,\dots,N$.
\end{remark}

\section{Numerical results}\label{numerics}

In the remaining part of the paper we shall refer to the system 
\eqref{23}, \eqref{23b} with the artificial diffusion matrix defined by 
\eqref{asm-bij}, \eqref{asm-bii} and \eqref{def-beta_ij_1}--\eqref{def-beta_ij_5} as to the 
Monotone Upwind-type Algebraically Stabilized (MUAS) method.
The AFC scheme with the Kuzmin limiter formulated in
Sect.~\ref{s6} will be simply called AFC scheme in the following. Our aim will
now be to compare the AFC scheme with the MUAS method numerically for the finite element
discretizations of \eqref{strong-steady} presented in Sect.~\ref{s2}. If not
stated otherwise, the bilinear form \eqref{eq_form_a} will be considered in
the discrete problem.

Under condition \eqref{assumption_min}, the only difference between the MUAS method and the
AFC scheme consists in the definition of $Q_i^\pm$, cf.~\eqref{def-beta_ij_2}
and \eqref{46b}. Our numerical experiments show that the difference between 
the results of the two methods is very small in this case. Since numerical
results for the AFC scheme under condition \eqref{assumption_min} have been
reported in many other papers, we shall concentrate on cases where condition
\eqref{assumption_min} is not satisfied.

As discussed in Sect.~\ref{s6}, condition \eqref{assumption_min} may be
violated if the triangulation $\TT_h$ is not of Delaunay type or if the
reaction coefficient $c$ is sufficiently large in comparison with
$\varepsilon$ and $\|\bb\|$. We shall start with a
reaction-dominated problem formulated in the following example.

\begin{example} (Reaction-dominated problem) \label{ex:reaction}
Problem \eqref{strong-steady} is considered with $\Omega = (0,1)^2$,
$\varepsilon=10^{-8}$, $\bb=(0.004,0.012)^T$, $c=g=1$, and $u_b=0$.
\end{example}

A natural question is why not to set simply $\bb=\bold0$ in
Example~\ref{ex:reaction}. However, this would lead to 
a symmetric matrix $(a_{ij})_{i,j=1}^N$ and since the AFC scheme is not
uniquely defined if $a_{ij}=a_{ji}>0$ for some indices $i\neq j$, it would be
difficult to interpret the results. Note also that since $c$ and $g$ are 
constant in
Example~\ref{ex:reaction}, equation~\eqref{strong-steady} can be reformulated
into a form with vanishing right-hand side. Indeed, if $u$ solves
\eqref{strong-steady}, then $(u-1)$ solves \eqref{strong-steady} with $g$
replaced by $0$ and $u_b=-1$. Then the maximum principles \eqref{eq_max_1a},
\eqref{eq_max_1b} with $G=\Omega$ imply that $(u-1)\in[-1,0]$ and hence
$u\in[0,1]$ in $\Omega$. The solution of \eqref{strong-steady} satisfies
$u\approx1$ away from layers which are located around the boundary of $\Omega$.

We will present results obtained on a uniform triangulation of the type
depicted on the left of Fig.~\ref{fig:grids} containing $21\times21$ vertices.
\begin{figure}[t]
\centerline{
\includegraphics[width=0.2\textwidth]{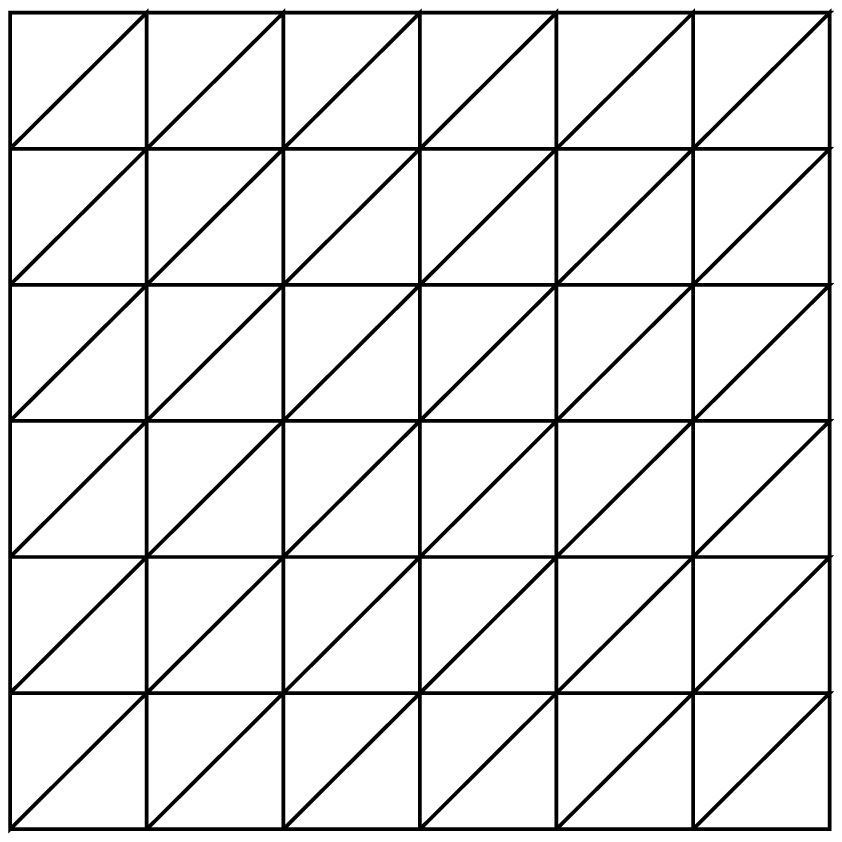}\hspace*{5ex}
\includegraphics[width=0.2\textwidth]{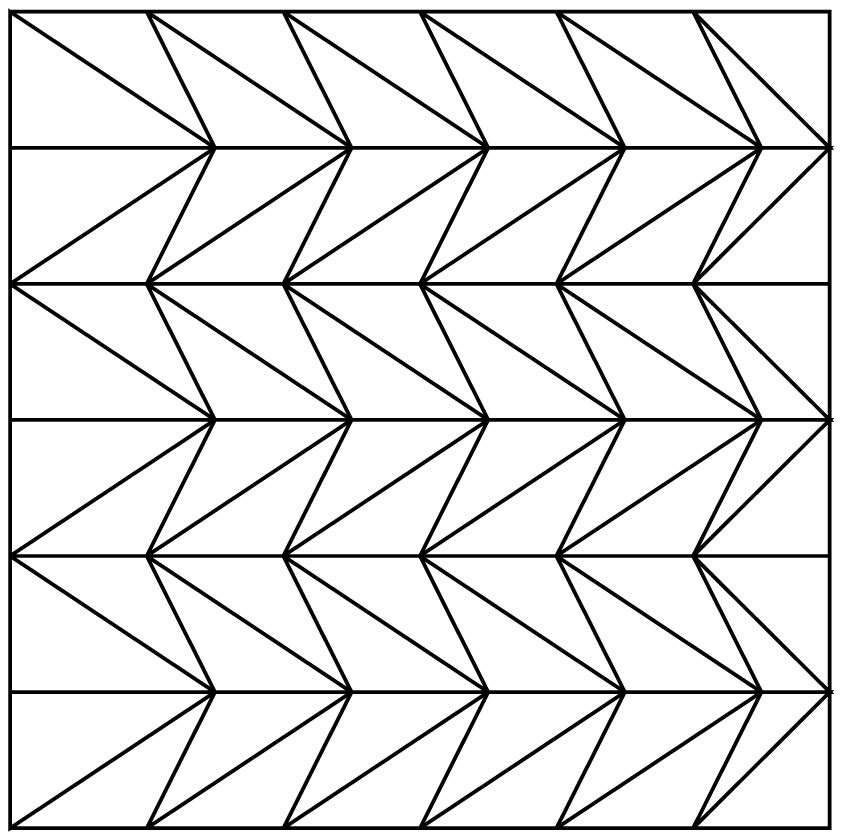}\hspace*{5ex}
\includegraphics[width=0.2\textwidth]{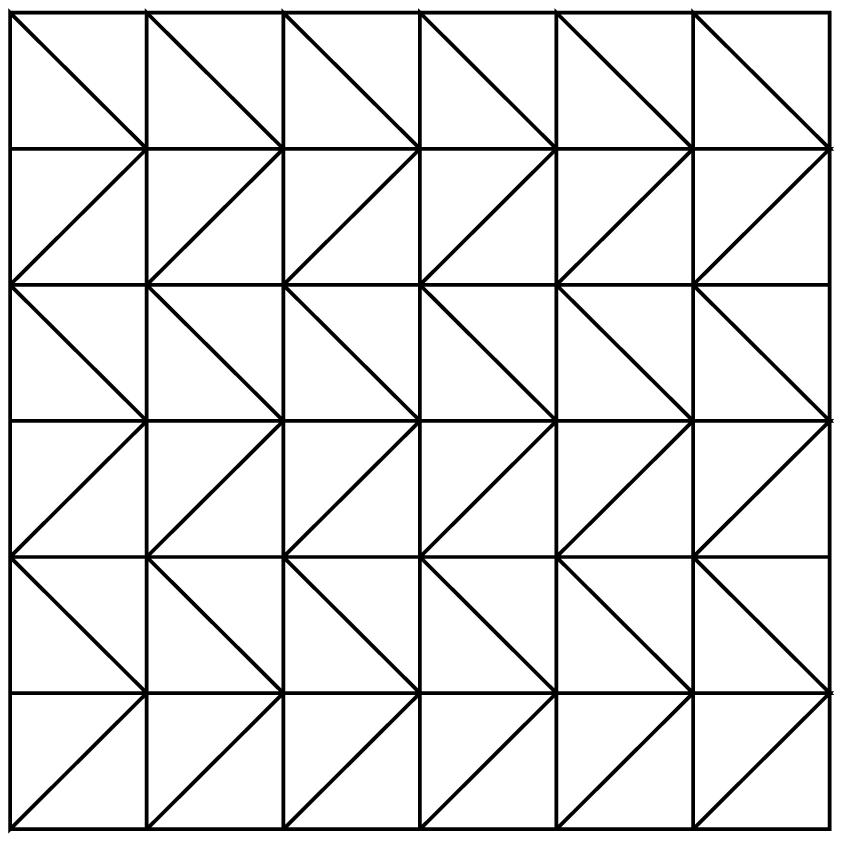}
}
\caption{Types of triangulations considered in numerical experiments}
\label{fig:grids}
\end{figure}
Then the matrix $(a_{ij})_{i,j=1}^N$ defined by \eqref{13} with $a_h=a$ has
only nonnegative entries and condition \eqref{assumption_min} is not satisfied.
The AFC scheme does not satisfy the DMP and provides a nonphysical solution,
see Fig.~\ref{fig:ex1} (left).
\begin{figure}[t]
\centerline{
\includegraphics[width=0.36\textwidth]{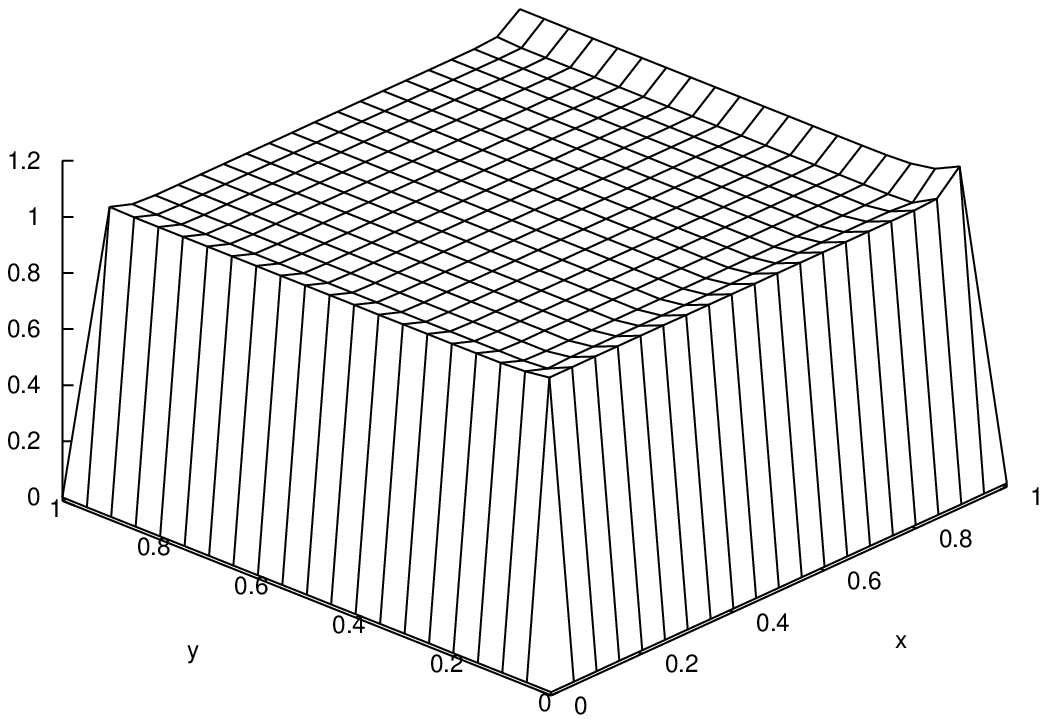}\hspace*{-3ex}
\includegraphics[width=0.36\textwidth]{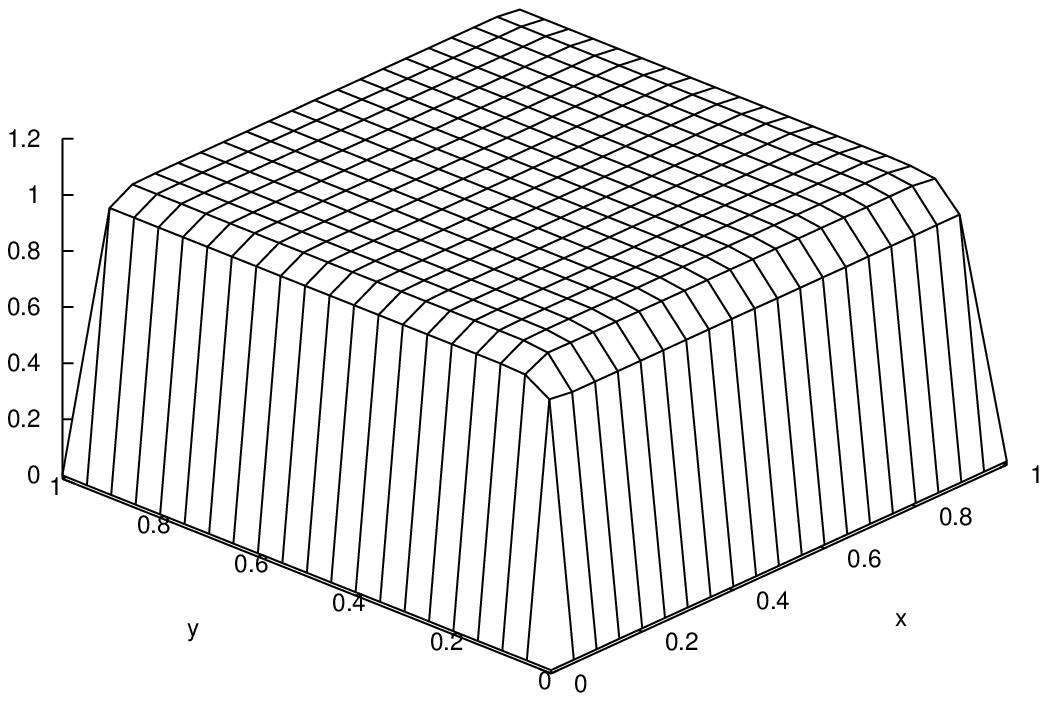}\hspace*{-3ex}
\includegraphics[width=0.36\textwidth]{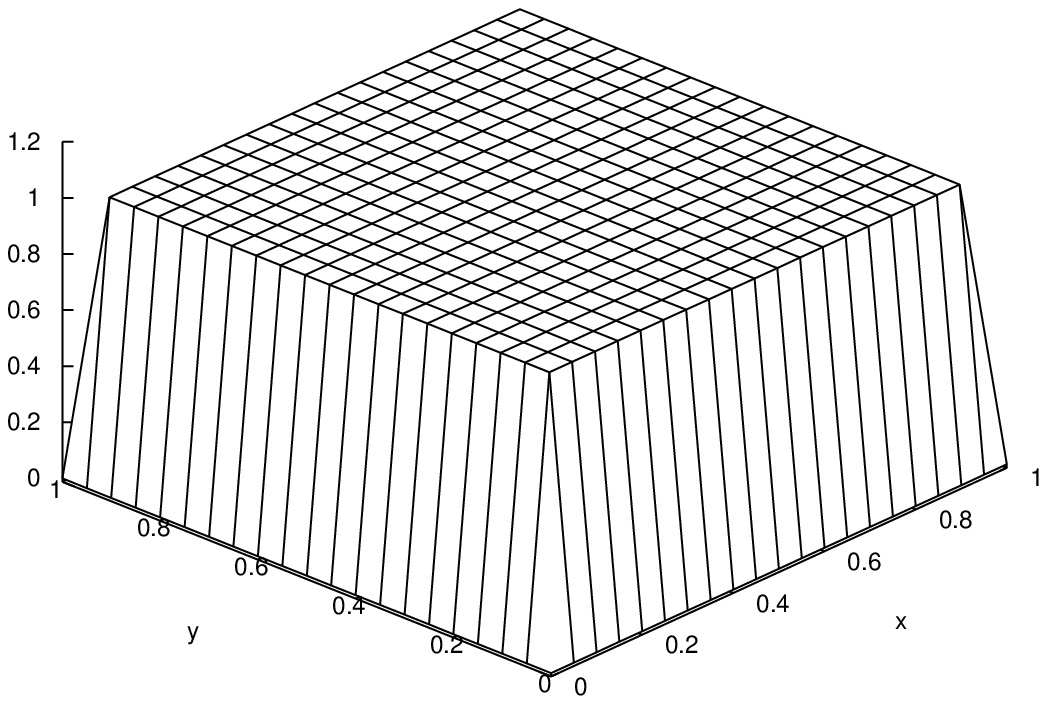}
}
\caption{Approximate solutions of Example~\ref{ex:reaction} computed on a
triangulation of the type shown on the left of Fig.~\ref{fig:grids}: 
AFC method (left), AFC method with lumping (middle), MUAS method (right)}\label{fig:ex1}
\end{figure}
As discussed in Sect.~\ref{s6}, a possible remedy is to define the bilinear
form $a_h$ by \eqref{eq_lumped_react}, i.e., to consider a lumping of the
reaction term. This provides a physically consistent approximate solution but
may lead to a smearing of the layers, see Fig.~\ref{fig:ex1} (middle). On the
other hand, applying the MUAS method, one obtains a very accurate solution with sharp
layers, see Fig.~\ref{fig:ex1} (right).

\begin{example} (Convection-dominated problem) 
\label{ex:outflow_layers}
Problem \eqref{strong-steady} is considered with $\Omega = (0,1)^2$,
$\varepsilon=10^{-2}$, $\bb=(\cos(-\pi/3),\sin(-\pi/3))^T$, $c=g=0$, and
$$
u_b(x,y) = \left\{ \begin{array}{ll}
0 & \quad\mbox{for $x=1$ or $y=0$,}\\
1 & \quad\mbox{else.}
\end{array}\right.
$$
To satisfy the assumptions on problem \eqref{strong-steady}, the 
discontinuous function $u_b$ can be replaced by a smooth function such that
the approximate solutions do not change for the triangulation considered in the
numerical experiments.
\end{example}

This example will be used to demonstrate that the AFC scheme can lead to
physically inconsistent solutions also in the convection-dominated case. To
this end, one has to use a triangulation which is not of Delaunay type. We
again consider a triangulation containing $21\times21$ vertices which is now
obtained from a triangulation of the type depicted on the right in
Fig.~\ref{fig:grids} by shifting interior nodes to the right by half
of the horizontal mesh width on each even horizontal mesh line. This gives 
a triangulation of the type shown in the middle of Fig.~\ref{fig:grids} for
which condition \eqref{assumption_min} is again not satisfied.
Like in Fig.~\ref{fig:ex1}, the results will be visualized using a uniform
square mesh having the same number of vertices (and hence also the same
horizontal mesh lines) as the mentioned triangulation.

According to the maximum principles \eqref{eq_max_2a}, \eqref{eq_max_2b}, the
solution of \eqref{strong-steady} with the data specified in
Example~\ref{ex:outflow_layers} satisfies $u\in[0,1]$ in $\Omega$.
Fig.~\ref{fig:ex2} (left) shows that this property is not preserved by the AFC 
scheme
\begin{figure}[t]%
\centerline{
\includegraphics[width=0.36\textwidth]{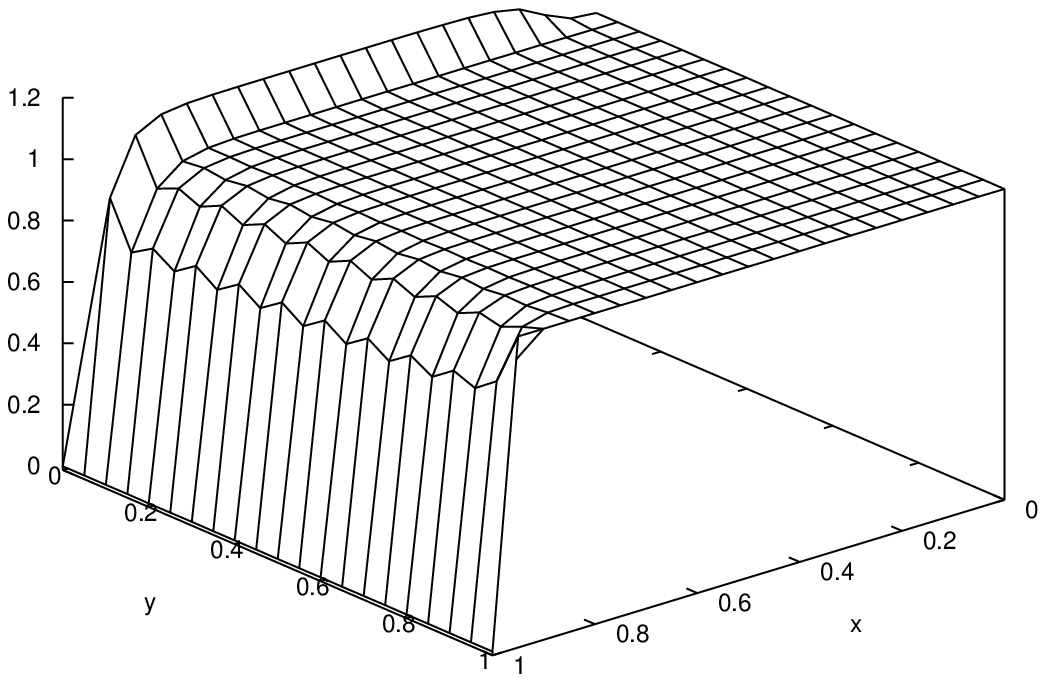}\hspace*{5ex}
\includegraphics[width=0.36\textwidth]{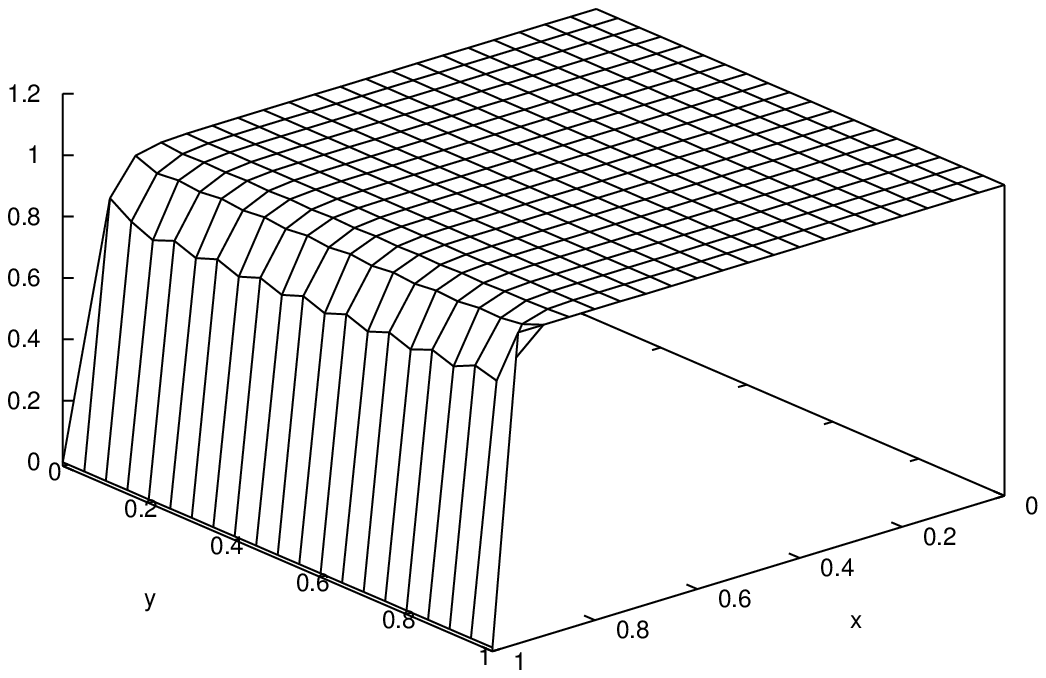}
}
\caption{Approximate solutions of Example~\ref{ex:outflow_layers} computed on a
triangulation of the type shown in the middle of Fig.~\ref{fig:grids}: 
AFC method (left), MUAS method (right)}\label{fig:ex2}
\end{figure}
for which the approximate solution contains a significant overshoot along the
line $y=0$. On the other hand, the MUAS method provides a qualitatively correct
approximate solution respecting the DMP, see Fig.~\ref{fig:ex2} (right).

\begin{example} (Diffusion-dominated problem) \label{ex:smooth}
Problem \eqref{strong-steady} is considered with $\Omega = (0,1)^2$,
$\varepsilon=10$, $\bb = (3,2)^T$, $c=1$, $u_b=0$, and the right-hand side $g$ 
chosen so that
$$
u(x,y) = 100\,x^2\,(1-x)^2\,y\,(1-y)\,(1-2y)
$$
is the solution of \eqref{strong-steady}.
\end{example}

In \cite{BJK16}, this example was considered on triangulations constructed 
similarly as the one in the middle of Fig.~\ref{fig:grids}; the difference was
that the shift of the respective interior nodes was only the tenth of the 
horizontal mesh width. It was observed that the convergence orders of the AFC
scheme with respect to various norms tend to zero if fine meshes are used. This 
behavior is even more pronounced on meshes of the type shown in the middle of
Fig.~\ref{fig:grids} (where the shift of the nodes is the half of the
horizontal mesh width), see Table~\ref{tab1}. In the tables, the value of $ne$ represents the number of edges along one horizontal mesh line (thus,
$ne=6$ for the meshes in Fig.~\ref{fig:grids}). Note that a lumping of the
reaction term has no significant influence on the results in this case. On the
other hand, applying the MUAS method, one observes a convergence in all the norms, see
Table~\ref{tab2}. This behavior is connected with the fact that the definition
of $Q_i^\pm$ was changed from \eqref{46b} to \eqref{def-beta_ij_2}. If the
original definition \eqref{46b} is used in the MUAS method, then the accuracy
deteriorates and the convergence orders tend to zero on fine meshes, see
Table~\ref{tab3}. Nevertheless, the convergence may fail also for the MUAS method when
too distorted meshes are considered. An example is given in Table~\ref{tab4},
where the results were computed on triangulations obtained from those of the
type depicted on the right in Fig.~\ref{fig:grids} by shifting the respective
interior nodes by eight tenths of the horizontal mesh width. However, also in 
this case the results are more accurate than in case of the AFC scheme. 

A possible explanation of the observed deteriorations of convergence orders is 
the loss of the linearity preservation when using certain non-Delaunay meshes.
Let us recall that the scheme \eqref{23} is called linearity preserving if
${\mathbb B}({\rm U})$ vanishes for any vector ${\rm U}$ representing a linear 
function in $\Omega$. Under further assumptions, this property enables to prove 
improved error estimates, see, e.g., \cite{BJK18}. It can be verified, that, in 
case of Table~\ref{tab2}, the MUAS method is linearity preserving, which is not true for
the schemes used to compute the results in Tables \ref{tab1}, \ref{tab3}, and 
\ref{tab4}. This could also explain why the replacement of \eqref{def-beta_ij_2}
by \eqref{46b} leads to the deterioration of the results since the absolute
values of $Q_i^\pm$ given by \eqref{46b} are smaller or equal to those given 
by \eqref{def-beta_ij_2} and hence the linearity preservation is more likely 
to hold if \eqref{def-beta_ij_2} is used.

\begin{table}[h]
\begin{center}
\begin{minipage}{274pt}
\caption{Errors and convergence orders of approximate solutions of 
Example~\ref{ex:smooth} computed using the AFC scheme on triangulations of the 
type shown in the middle of Fig.~\ref{fig:grids}}\label{tab1}
\begin{tabular}{@{}rcccccc@{}}
\toprule
$ne$ & $\|u-u_h\|_{0,\Omega}^{}$ & order &   $\vert u-u_h\vert_{1,\Omega}^{}$ &
order & $\|u-u_h\|_h^{}$ & order\\
\midrule
  16 &  5.636e$-$2 &  0.22 & 6.741e$-$1 &  0.41 & 2.626e$+$0 &  0.24\\
  32 &  5.384e$-$2 &  0.07 & 5.908e$-$1 &  0.19 & 2.437e$+$0 &  0.11\\
  64 &  5.332e$-$2 &  0.01 & 5.661e$-$1 &  0.06 & 2.380e$+$0 &  0.03\\
 128 &  5.321e$-$2 &  0.00 & 5.593e$-$1 &  0.02 & 2.363e$+$0 &  0.01\\
 256 &  5.319e$-$2 &  0.00 & 5.575e$-$1 &  0.00 & 2.358e$+$0 &  0.00\\
 512 &  5.320e$-$2 &  0.00 & 5.570e$-$1 &  0.00 & 2.356e$+$0 &  0.00\\
1024 &  5.321e$-$2 &  0.00 & 5.568e$-$1 &  0.00 & 2.356e$+$0 &  0.00\\
\botrule
\end{tabular}
\end{minipage}
\end{center}
\end{table}
\begin{table}[h]
\begin{center}
\begin{minipage}{274pt}
\caption{Errors and convergence orders of approximate solutions of 
Example~\ref{ex:smooth} computed using the MUAS method on triangulations of the 
type shown in the middle of Fig.~\ref{fig:grids}}\label{tab2}
\begin{tabular}{@{}rcccccc@{}}
\toprule
$ne$ & $\|u-u_h\|_{0,\Omega}^{}$ & order &   $\vert u-u_h\vert_{1,\Omega}^{}$ &
order & $\|u-u_h\|_h^{}$ & order\\
\midrule
  16 &  2.206e$-$2 &  1.60 & 4.847e$-$1 &  0.86 & 1.581e$+$0 &  0.88\\
  32 &  6.967e$-$3 &  1.66 & 2.505e$-$1 &  0.95 & 8.038e$-$1 &  0.98\\
  64 &  2.249e$-$3 &  1.63 & 1.263e$-$1 &  0.99 & 4.034e$-$1 &  0.99\\
 128 &  7.770e$-$4 &  1.53 & 6.287e$-$2 &  1.01 & 2.003e$-$1 &  1.01\\
 256 &  2.471e$-$4 &  1.65 & 3.115e$-$2 &  1.01 & 9.904e$-$2 &  1.02\\
 512 &  7.108e$-$5 &  1.80 & 1.544e$-$2 &  1.01 & 4.901e$-$2 &  1.02\\
1024 &  1.915e$-$5 &  1.89 & 7.677e$-$3 &  1.01 & 2.433e$-$2 &  1.01\\
\botrule
\end{tabular}
\end{minipage}
\end{center}
\end{table}
\begin{table}[h]
\begin{center}
\begin{minipage}{274pt}
\caption{Errors and convergence orders of approximate solutions of 
Example~\ref{ex:smooth} computed using the MUAS method with $Q_i^\pm$ defined by 
\eqref{46b} instead of \eqref{def-beta_ij_2}. The used triangulations are
of the type shown in the middle of Fig.~\ref{fig:grids}}\label{tab3}
\begin{tabular}{@{}rcccccc@{}}
\toprule
$ne$ & $\|u-u_h\|_{0,\Omega}^{}$ & order &   $\vert u-u_h\vert_{1,\Omega}^{}$ &
order & $\|u-u_h\|_h^{}$ & order\\
\midrule
  16 &  7.677e$-$2 &  0.42 & 7.526e$-$1 &  0.40 & 3.019e$+$0 &  0.28\\
  32 &  6.399e$-$2 &  0.26 & 6.382e$-$1 &  0.24 & 2.657e$+$0 &  0.18\\
  64 &  5.806e$-$2 &  0.14 & 5.903e$-$1 &  0.11 & 2.488e$+$0 &  0.09\\
 128 &  5.543e$-$2 &  0.07 & 5.711e$-$1 &  0.05 & 2.415e$+$0 &  0.04\\
 256 &  5.426e$-$2 &  0.03 & 5.632e$-$1 &  0.02 & 2.383e$+$0 &  0.02\\
 512 &  5.372e$-$2 &  0.01 & 5.598e$-$1 &  0.01 & 2.369e$+$0 &  0.01\\
1024 &  5.346e$-$2 &  0.01 & 5.582e$-$1 &  0.00 & 2.362e$+$0 &  0.00\\
\botrule
\end{tabular}
\end{minipage}
\end{center}
\end{table}
\begin{table}[h]
\begin{center}
\begin{minipage}{274pt}
\caption{Errors and convergence orders of approximate solutions of 
Example~\ref{ex:smooth} computed using the MUAS method on triangulations of the 
type depicted in the middle of Fig.~\ref{fig:grids} obtained by shifting the 
respective interior nodes by eight tenths of the horizontal mesh width}
\label{tab4}
\begin{tabular}{@{}rcccccc@{}}
\toprule
$ne$ & $\|u-u_h\|_{0,\Omega}^{}$ & order &   $\vert u-u_h\vert_{1,\Omega}^{}$ &
order & $\|u-u_h\|_h^{}$ & order\\
\midrule
  16 &  4.589e$-$2 &  1.08 & 6.405e$-$1 &  0.70 & 2.303e$+$0 &  0.72\\
  32 &  2.528e$-$2 &  0.86 & 3.834e$-$1 &  0.74 & 1.326e$+$0 &  0.80\\
  64 &  1.714e$-$2 &  0.56 & 2.442e$-$1 &  0.65 & 8.316e$-$1 &  0.67\\
 128 &  1.347e$-$2 &  0.35 & 1.758e$-$1 &  0.47 & 5.948e$-$1 &  0.48\\
 256 &  1.178e$-$2 &  0.19 & 1.468e$-$1 &  0.26 & 4.956e$-$1 &  0.26\\
 512 &  1.100e$-$2 &  0.10 & 1.355e$-$1 &  0.12 & 4.576e$-$1 &  0.12\\
1024 &  1.062e$-$2 &  0.05 & 1.311e$-$1 &  0.05 & 4.428e$-$1 &  0.05\\
\botrule
\end{tabular}
\end{minipage}
\end{center}
\end{table}

\begin{remark} Comprehensive numerical studies of the MUAS method and, in particular, 
comparisons with the AFC schemes with Kuzmin limiter and with BJK limiter can be found in 
\cite{JJK21}. In this paper, the behavior of these methods on adaptively
refined meshes, 
with conforming closure or with hanging vertices, is studied. The assessment focuses on the 
satisfaction of the global DMP, the accuracy of the numerical solutions, and the efficiency 
of the solver for the arising nonlinear problems. 
\end{remark}


\newpage

\begin{thebibliography}{40}
\ifx \bisbn   \undefined \def \bisbn  #1{ISBN #1}\fi
\ifx \binits  \undefined \def \binits#1{#1}\fi
\ifx \bauthor  \undefined \def \bauthor#1{#1}\fi
\ifx \batitle  \undefined \def \batitle#1{#1}\fi
\ifx \bjtitle  \undefined \def \bjtitle#1{#1}\fi
\ifx \bvolume  \undefined \def \bvolume#1{\textbf{#1}}\fi
\ifx \byear  \undefined \def \byear#1{#1}\fi
\ifx \bissue  \undefined \def \bissue#1{#1}\fi
\ifx \bfpage  \undefined \def \bfpage#1{#1}\fi
\ifx \blpage  \undefined \def \blpage #1{#1}\fi
\ifx \burl  \undefined \def \burl#1{\textsf{#1}}\fi
\ifx \doiurl  \undefined \def \doiurl#1{\url{https://doi.org/#1}}\fi
\ifx \betal  \undefined \def \betal{\textit{et al.}}\fi
\ifx \binstitute  \undefined \def \binstitute#1{#1}\fi
\ifx \binstitutionaled  \undefined \def \binstitutionaled#1{#1}\fi
\ifx \bctitle  \undefined \def \bctitle#1{#1}\fi
\ifx \beditor  \undefined \def \beditor#1{#1}\fi
\ifx \bpublisher  \undefined \def \bpublisher#1{#1}\fi
\ifx \bbtitle  \undefined \def \bbtitle#1{#1}\fi
\ifx \bedition  \undefined \def \bedition#1{#1}\fi
\ifx \bseriesno  \undefined \def \bseriesno#1{#1}\fi
\ifx \blocation  \undefined \def \blocation#1{#1}\fi
\ifx \bsertitle  \undefined \def \bsertitle#1{#1}\fi
\ifx \bsnm \undefined \def \bsnm#1{#1}\fi
\ifx \bsuffix \undefined \def \bsuffix#1{#1}\fi
\ifx \bparticle \undefined \def \bparticle#1{#1}\fi
\ifx \barticle \undefined \def \barticle#1{#1}\fi
\bibcommenthead
\ifx \bconfdate \undefined \def \bconfdate #1{#1}\fi
\ifx \botherref \undefined \def \botherref #1{#1}\fi
\ifx \url \undefined \def \url#1{\textsf{#1}}\fi
\ifx \bchapter \undefined \def \bchapter#1{#1}\fi
\ifx \bbook \undefined \def \bbook#1{#1}\fi
\ifx \bcomment \undefined \def \bcomment#1{#1}\fi
\ifx \oauthor \undefined \def \oauthor#1{#1}\fi
\ifx \citeauthoryear \undefined \def \citeauthoryear#1{#1}\fi
\ifx \endbibitem  \undefined \def \endbibitem {}\fi
\ifx \bconflocation  \undefined \def \bconflocation#1{#1}\fi
\ifx \arxivurl  \undefined \def \arxivurl#1{\textsf{#1}}\fi
\csname PreBibitemsHook\endcsname

\bibitem{ACF+11}
\begin{barticle}
\bauthor{\bsnm{Augustin}, \binits{M.}},
\bauthor{\bsnm{Caiazzo}, \binits{A.}},
\bauthor{\bsnm{Fiebach}, \binits{A.}},
\bauthor{\bsnm{Fuhrmann}, \binits{J.}},
\bauthor{\bsnm{John}, \binits{V.}},
\bauthor{\bsnm{Linke}, \binits{A.}},
\bauthor{\bsnm{Umla}, \binits{R.}}:
\batitle{An assessment of discretizations for convection-dominated
  convection--diffusion equations}.
\bjtitle{Comput.~Methods Appl.~Mech.~Engrg.}
\bvolume{200}(\bissue{47-48}),
\bfpage{3395}--\blpage{3409}
(\byear{2011})
\end{barticle}
\endbibitem

\bibitem{BT81}
\begin{barticle}
\bauthor{\bsnm{Baba}, \binits{K.}},
\bauthor{\bsnm{Tabata}, \binits{M.}}:
\batitle{On a conservative upwind finite element scheme for convective
  diffusion equations}.
\bjtitle{RAIRO Anal. Num\'er.}
\bvolume{15}(\bissue{1}),
\bfpage{3}--\blpage{25}
(\byear{1981})
\end{barticle}
\endbibitem

\bibitem{BB17}
\begin{barticle}
\bauthor{\bsnm{Badia}, \binits{S.}},
\bauthor{\bsnm{Bonilla}, \binits{J.}}:
\batitle{Monotonicity-preserving finite element schemes based on differentiable
  nonlinear stabilization}.
\bjtitle{Comput. Methods Appl. Mech. Engrg.}
\bvolume{313},
\bfpage{133}--\blpage{158}
(\byear{2017})
\end{barticle}
\endbibitem

\bibitem{BBK17a}
\begin{barticle}
\bauthor{\bsnm{Barrenechea}, \binits{G.R.}},
\bauthor{\bsnm{Burman}, \binits{E.}},
\bauthor{\bsnm{Karakatsani}, \binits{F.}}:
\batitle{Blending low-order stabilised finite element methods: {A}
  positivity-preserving local projection method for the convection--diffusion
  equation}.
\bjtitle{Comput. Methods Appl. Mech. Engrg.}
\bvolume{317},
\bfpage{1169}--\blpage{1193}
(\byear{2017})
\end{barticle}
\endbibitem

\bibitem{BBK17}
\begin{barticle}
\bauthor{\bsnm{Barrenechea}, \binits{G.R.}},
\bauthor{\bsnm{Burman}, \binits{E.}},
\bauthor{\bsnm{Karakatsani}, \binits{F.}}:
\batitle{Edge-based nonlinear diffusion for finite element approximations of
  convection--diffusion equations and its relation to algebraic flux-correction
  schemes}.
\bjtitle{Numer.~Math.}
\bvolume{135}(\bissue{2}),
\bfpage{521}--\blpage{545}
(\byear{2017})
\end{barticle}
\endbibitem

\bibitem{BJK15}
\begin{barticle}
\bauthor{\bsnm{Barrenechea}, \binits{G.R.}},
\bauthor{\bsnm{John}, \binits{V.}},
\bauthor{\bsnm{Knobloch}, \binits{P.}}:
\batitle{Some analytical results for an algebraic flux correction scheme for a
  steady convection--diffusion equation in one dimension}.
\bjtitle{IMA J. Numer. Anal.}
\bvolume{35}(\bissue{4}),
\bfpage{1729}--\blpage{1756}
(\byear{2015})
\end{barticle}
\endbibitem

\bibitem{BJK16}
\begin{barticle}
\bauthor{\bsnm{Barrenechea}, \binits{G.R.}},
\bauthor{\bsnm{John}, \binits{V.}},
\bauthor{\bsnm{Knobloch}, \binits{P.}}:
\batitle{Analysis of algebraic flux correction schemes}.
\bjtitle{SIAM J.~Numer.~Anal.}
\bvolume{54}(\bissue{4}),
\bfpage{2427}--\blpage{2451}
(\byear{2016})
\end{barticle}
\endbibitem

\bibitem{BJK17}
\begin{barticle}
\bauthor{\bsnm{Barrenechea}, \binits{G.R.}},
\bauthor{\bsnm{John}, \binits{V.}},
\bauthor{\bsnm{Knobloch}, \binits{P.}}:
\batitle{An algebraic flux correction scheme satisfying the discrete maximum
  principle and linearity preservation on general meshes}.
\bjtitle{Math.~Models Methods Appl.~Sci.}
\bvolume{27}(\bissue{3}),
\bfpage{525}--\blpage{548}
(\byear{2017})
\end{barticle}
\endbibitem

\bibitem{BJK18}
\begin{barticle}
\bauthor{\bsnm{Barrenechea}, \binits{G.R.}},
\bauthor{\bsnm{John}, \binits{V.}},
\bauthor{\bsnm{Knobloch}, \binits{P.}},
\bauthor{\bsnm{Rankin}, \binits{R.}}:
\batitle{A unified analysis of algebraic flux correction schemes for
  convection--diffusion equations}.
\bjtitle{SeMA J.}
\bvolume{75}(\bissue{4}),
\bfpage{655}--\blpage{685}
(\byear{2018})
\end{barticle}
\endbibitem

\bibitem{BorisBook73}
\begin{barticle}
\bauthor{\bsnm{Boris}, \binits{J.P.}},
\bauthor{\bsnm{Book}, \binits{D.L.}}:
\batitle{Flux-corrected transport. {I}. {S}{H}{A}{S}{T}{A}, a fluid transport
  algorithm that works}.
\bjtitle{J.~Comput.~Phys.}
\bvolume{11}(\bissue{1}),
\bfpage{38}--\blpage{69}
(\byear{1973})
\end{barticle}
\endbibitem

\bibitem{BE02}
\begin{barticle}
\bauthor{\bsnm{Burman}, \binits{E.}},
\bauthor{\bsnm{Ern}, \binits{A.}}:
\batitle{Nonlinear diffusion and discrete maximum principle for stabilized
  {G}alerkin approximations of the convection--diffusion-reaction equation}.
\bjtitle{Comput. Methods Appl. Mech. Engrg.}
\bvolume{191}(\bissue{35}),
\bfpage{3833}--\blpage{3855}
(\byear{2002})
\end{barticle}
\endbibitem

\bibitem{BE05}
\begin{barticle}
\bauthor{\bsnm{Burman}, \binits{E.}},
\bauthor{\bsnm{Ern}, \binits{A.}}:
\batitle{Stabilized {G}alerkin approximation of convection--diffusion--reaction
  equations: discrete maximum principle and convergence}.
\bjtitle{Math.~Comp.}
\bvolume{74}(\bissue{252}),
\bfpage{1637}--\blpage{1652}
(\byear{2005})
\end{barticle}
\endbibitem

\bibitem{BH04}
\begin{barticle}
\bauthor{\bsnm{Burman}, \binits{E.}},
\bauthor{\bsnm{Hansbo}, \binits{P.}}:
\batitle{Edge stabilization for {G}alerkin approximations of
  convection--diffusion--reaction problems}.
\bjtitle{Comput.~Methods Appl.~Mech.~Engrg.}
\bvolume{193}(\bissue{15-16}),
\bfpage{1437}--\blpage{1453}
(\byear{2004})
\end{barticle}
\endbibitem

\bibitem{Ciarlet}
\begin{bbook}
\bauthor{\bsnm{Ciarlet}, \binits{P.G.}}:
\bbtitle{The Finite Element Method for Elliptic Problems}.
\bpublisher{North-Holland},
\blocation{Amsterdam}
(\byear{1978})
\end{bbook}
\endbibitem

\bibitem{Evans}
\begin{bbook}
\bauthor{\bsnm{Evans}, \binits{L.C.}}:
\bbtitle{Partial Differential Equations},
\bedition{2}nd edn.
\bpublisher{American Mathematical Society},
\blocation{Providence, RI}
(\byear{2010})
\end{bbook}
\endbibitem

\bibitem{GT01}
\begin{bbook}
\bauthor{\bsnm{Gilbarg}, \binits{D.}},
\bauthor{\bsnm{Trudinger}, \binits{N.S.}}:
\bbtitle{Elliptic Partial Differential Equations of Second Order}.
\bpublisher{Springer},
\blocation{Berlin}
(\byear{2001})
\end{bbook}
\endbibitem

\bibitem{GNPY14}
\begin{barticle}
\bauthor{\bsnm{Guermond}, \binits{J.-L.}},
\bauthor{\bsnm{Nazarov}, \binits{M.}},
\bauthor{\bsnm{Popov}, \binits{B.}},
\bauthor{\bsnm{Yang}, \binits{Y.}}:
\batitle{A second-order maximum principle preserving {L}agrange finite element
  technique for nonlinear scalar conservation equations}.
\bjtitle{SIAM J. Numer. Anal.}
\bvolume{52}(\bissue{4}),
\bfpage{2163}--\blpage{2182}
(\byear{2014})
\end{barticle}
\endbibitem

\bibitem{GKT12}
\begin{barticle}
\bauthor{\bsnm{Gurris}, \binits{M.}},
\bauthor{\bsnm{Kuzmin}, \binits{D.}},
\bauthor{\bsnm{Turek}, \binits{S.}}:
\batitle{Implicit finite element schemes for the stationary compressible
  {E}uler equations}.
\bjtitle{Internat. J. Numer. Methods Fluids}
\bvolume{69}(\bissue{1}),
\bfpage{1}--\blpage{28}
(\byear{2012})
\end{barticle}
\endbibitem

\bibitem{JJ19}
\begin{barticle}
\bauthor{\bsnm{Jha}, \binits{A.}},
\bauthor{\bsnm{John}, \binits{V.}}:
\batitle{A study of solvers for nonlinear {AFC} discretizations of
  convection--diffusion equations}.
\bjtitle{Comput. Math. Appl.}
\bvolume{78}(\bissue{9}),
\bfpage{3117}--\blpage{3138}
(\byear{2019})
\end{barticle}
\endbibitem

\bibitem{JJK21}
\begin{botherref}
\oauthor{\bsnm{Jha}, \binits{A.}},
\oauthor{\bsnm{John}, \binits{V.}},
\oauthor{\bsnm{Knobloch}, \binits{P.}}:
Adaptive grids in the context of algebraic stabilizations for
  convection--diffusion--reaction equations.
In preparation (2021)
\end{botherref}
\endbibitem

\bibitem{JK08}
\begin{barticle}
\bauthor{\bsnm{John}, \binits{V.}},
\bauthor{\bsnm{Knobloch}, \binits{P.}}:
\batitle{On spurious oscillations at layers diminishing ({SOLD}) methods for
  convection--diffusion equations: {P}art {II} -- {A}nalysis for {$P_1$} and
  {$Q_1$} finite elements}.
\bjtitle{Comput.~Methods Appl.~Mech.~Engrg.}
\bvolume{197}(\bissue{21-24}),
\bfpage{1997}--\blpage{2014}
(\byear{2008})
\end{barticle}
\endbibitem

\bibitem{JS08}
\begin{barticle}
\bauthor{\bsnm{John}, \binits{V.}},
\bauthor{\bsnm{Schmeyer}, \binits{E.}}:
\batitle{Finite element methods for time-dependent
  convection--diffusion--reaction equations with small diffusion}.
\bjtitle{Comput.~Methods Appl.~Mech.~Engrg.}
\bvolume{198}(\bissue{3-4}),
\bfpage{475}--\blpage{494}
(\byear{2008})
\end{barticle}
\endbibitem

\bibitem{Knobloch06}
\begin{barticle}
\bauthor{\bsnm{Knobloch}, \binits{P.}}:
\batitle{Improvements of the {M}izukami--{H}ughes method for
  convection--diffusion equations}.
\bjtitle{Comput. Methods Appl. Mech. Engrg.}
\bvolume{196}(\bissue{1-3}),
\bfpage{579}--\blpage{594}
(\byear{2006})
\end{barticle}
\endbibitem

\bibitem{Kno10}
\begin{barticle}
\bauthor{\bsnm{Knobloch}, \binits{P.}}:
\batitle{Numerical solution of convection--diffusion equations using a
  nonlinear method of upwind type}.
\bjtitle{J.~Sci.~Comput.}
\bvolume{43}(\bissue{3}),
\bfpage{454}--\blpage{470}
(\byear{2010})
\end{barticle}
\endbibitem

\bibitem{Kno17}
\begin{bchapter}
\bauthor{\bsnm{Knobloch}, \binits{P.}}:
\bctitle{On the discrete maximum principle for algebraic flux correction
  schemes with limiters of upwind type}.
In: \beditor{\bsnm{Huang}, \binits{Z.}},
\beditor{\bsnm{Stynes}, \binits{M.}},
\beditor{\bsnm{Zhang}, \binits{Z.}} (eds.)
\bbtitle{Boundary and Interior Layers, Computational and Asymptotic Methods
  BAIL 2016}.
\bsertitle{Lect. Notes Comput. Sci. Eng.},
vol. \bseriesno{120},
pp. \bfpage{129}--\blpage{139}.
\bpublisher{Springer},
\blocation{Cham}
(\byear{2017})
\end{bchapter}
\endbibitem

\bibitem{Kno19}
\begin{bchapter}
\bauthor{\bsnm{Knobloch}, \binits{P.}}:
\bctitle{A linearity preserving algebraic flux correction scheme of upwind type
  satisfying the discrete maximum principle on arbitrary meshes}.
In: \beditor{\bsnm{Radu}, \binits{F.A.}},
\beditor{\bsnm{Kumar}, \binits{K.}},
\beditor{\bsnm{Berre}, \binits{I.}},
\beditor{\bsnm{Nordbotten}, \binits{J.M.}},
\beditor{\bsnm{Pop}, \binits{I.S.}} (eds.)
\bbtitle{Numerical Mathematics and Advanced Applications ENUMATH 2017}.
\bsertitle{Lect. Notes Comput. Sci. Eng.},
vol. \bseriesno{126},
pp. \bfpage{909}--\blpage{918}.
\bpublisher{Springer},
\blocation{Cham}
(\byear{2019})
\end{bchapter}
\endbibitem

\bibitem{Kno21}
\begin{bchapter}
\bauthor{\bsnm{Knobloch}, \binits{P.}}:
\bctitle{A new algebraically stabilized method for
  convection--diffusion--reaction equations}.
In: \beditor{\bsnm{Vermolen}, \binits{F.J.}},
\beditor{\bsnm{Vuik}, \binits{C.}} (eds.)
\bbtitle{Numerical Mathematics and Advanced Applications {ENUMATH} 2019}.
\bsertitle{Lect. Notes Comput. Sci. Eng.},
vol. \bseriesno{139},
pp. \bfpage{605}--\blpage{613}.
\bpublisher{Springer},
\blocation{Cham}
(\byear{2021})
\end{bchapter}
\endbibitem

\bibitem{Kuzmin06}
\begin{barticle}
\bauthor{\bsnm{Kuzmin}, \binits{D.}}:
\batitle{On the design of general-purpose flux limiters for finite element
  schemes. {I}. {S}calar convection}.
\bjtitle{J.~Comput.~Phys.}
\bvolume{219}(\bissue{2}),
\bfpage{513}--\blpage{531}
(\byear{2006})
\end{barticle}
\endbibitem

\bibitem{Kuzmin07}
\begin{bchapter}
\bauthor{\bsnm{Kuzmin}, \binits{D.}}:
\bctitle{Algebraic flux correction for finite element discretizations of
  coupled systems}.
In: \beditor{\bsnm{Papadrakakis}, \binits{M.}},
\beditor{\bsnm{O{\~n}ate}, \binits{E.}},
\beditor{\bsnm{Schrefler}, \binits{B.}} (eds.)
\bbtitle{Proceedings of the Int.~Conf.~on Computational Methods for Coupled
  Problems in Science and Engineering},
pp. \bfpage{1}--\blpage{5}.
\bpublisher{CIMNE},
\blocation{Barcelona}
(\byear{2007})
\end{bchapter}
\endbibitem

\bibitem{Kuzmin09}
\begin{barticle}
\bauthor{\bsnm{Kuzmin}, \binits{D.}}:
\batitle{Explicit and implicit {FEM-FCT} algorithms with flux linearization}.
\bjtitle{J.~Comput.~Phys.}
\bvolume{228}(\bissue{7}),
\bfpage{2517}--\blpage{2534}
(\byear{2009})
\end{barticle}
\endbibitem

\bibitem{KuzminMoeller05}
\begin{bchapter}
\bauthor{\bsnm{Kuzmin}, \binits{D.}}:
\bctitle{Algebraic flux correction {I}. {S}calar conservation laws}.
In: \beditor{\bsnm{Kuzmin}, \binits{D.}},
\beditor{\bsnm{L\"ohner}, \binits{R.}},
\beditor{\bsnm{Turek}, \binits{S.}} (eds.)
\bbtitle{Flux-Corrected Transport. Principles, Algorithms, and Applications},
\bedition{2}nd edn.,
pp. \bfpage{145}--\blpage{192}.
\bpublisher{Springer},
\blocation{Dordrecht}
(\byear{2012})
\end{bchapter}
\endbibitem

\bibitem{Kuzmin12}
\begin{barticle}
\bauthor{\bsnm{Kuzmin}, \binits{D.}}:
\batitle{Linearity-preserving flux correction and convergence acceleration for
  constrained {G}alerkin schemes}.
\bjtitle{J.~Comput.~Appl.~Math.}
\bvolume{236}(\bissue{9}),
\bfpage{2317}--\blpage{2337}
(\byear{2012})
\end{barticle}
\endbibitem

\bibitem{KS17}
\begin{barticle}
\bauthor{\bsnm{Kuzmin}, \binits{D.}},
\bauthor{\bsnm{Shadid}, \binits{J.N.}}:
\batitle{Gradient-based nodal limiters for artificial diffusion operators in
  finite element schemes for transport equations}.
\bjtitle{Internat. J. Numer. Methods Fluids}
\bvolume{84}(\bissue{11}),
\bfpage{675}--\blpage{695}
(\byear{2017})
\end{barticle}
\endbibitem

\bibitem{KT04}
\begin{barticle}
\bauthor{\bsnm{Kuzmin}, \binits{D.}},
\bauthor{\bsnm{Turek}, \binits{S.}}:
\batitle{High-resolution {FEM}-{TVD} schemes based on a fully multidimensional
  flux limiter}.
\bjtitle{J. Comput. Phys.}
\bvolume{198}(\bissue{1}),
\bfpage{131}--\blpage{158}
(\byear{2004})
\end{barticle}
\endbibitem

\bibitem{Lohman19}
\begin{bbook}
\bauthor{\bsnm{Lohmann}, \binits{C.}}:
\bbtitle{Physics-compatible Finite Element Methods for Scalar and Tensorial
  Advection Problems}.
\bpublisher{Springer},
\blocation{Wiesbaden}
(\byear{2019})
\end{bbook}
\endbibitem

\bibitem{LKSM17}
\begin{barticle}
\bauthor{\bsnm{Lohmann}, \binits{C.}},
\bauthor{\bsnm{Kuzmin}, \binits{D.}},
\bauthor{\bsnm{Shadid}, \binits{J.N.}},
\bauthor{\bsnm{Mabuza}, \binits{S.}}:
\batitle{Flux-corrected transport algorithms for continuous {G}alerkin methods
  based on high order {B}ernstein finite elements}.
\bjtitle{J. Comput. Phys.}
\bvolume{344},
\bfpage{151}--\blpage{186}
(\byear{2017})
\end{barticle}
\endbibitem

\bibitem{MH85}
\begin{barticle}
\bauthor{\bsnm{Mizukami}, \binits{A.}},
\bauthor{\bsnm{Hughes}, \binits{T.J.R.}}:
\batitle{A {P}etrov--{G}alerkin finite element method for convection-dominated
  flows: an accurate upwinding technique for satisfying the maximum principle}.
\bjtitle{Comput. Methods Appl. Mech. Engrg.}
\bvolume{50}(\bissue{2}),
\bfpage{181}--\blpage{193}
(\byear{1985})
\end{barticle}
\endbibitem

\bibitem{RST08}
\begin{bbook}
\bauthor{\bsnm{Roos}, \binits{H.-G.}},
\bauthor{\bsnm{Stynes}, \binits{M.}},
\bauthor{\bsnm{Tobiska}, \binits{L.}}:
\bbtitle{Robust Numerical Methods for Singularly Perturbed Differential
  Equations. {C}onvection--Diffusion--Reaction and Flow Problems}.
\bpublisher{Springer},
\blocation{Berlin}
(\byear{2008})
\end{bbook}
\endbibitem

\bibitem{Temam77}
\begin{bbook}
\bauthor{\bsnm{Temam}, \binits{R.}}:
\bbtitle{{N}avier--{S}tokes Equations. {T}heory and Numerical Analysis}.
\bpublisher{North-Holland},
\blocation{Amsterdam}
(\byear{1977})
\end{bbook}
\endbibitem

\bibitem{Zalesak79}
\begin{barticle}
\bauthor{\bsnm{Zalesak}, \binits{S.T.}}:
\batitle{Fully multidimensional flux-corrected transport algorithms for
  fluids}.
\bjtitle{J.~Comput.~Phys.}
\bvolume{31}(\bissue{3}),
\bfpage{335}--\blpage{362}
(\byear{1979})
\end{barticle}
\endbibitem

\end{thebibliography}
\end{document}